\documentclass{amsart}
\usepackage{amsmath,latexsym,amssymb,amsfonts,hyperref}

\newtheorem{theorem}{Theorem}[section]
\newtheorem{lemma}[theorem]{Lemma}
\newtheorem{prop}[theorem]{Proposition}
\newtheorem{corollary}[theorem]{Corollary}

\theoremstyle{definition}
\newtheorem{remark}[theorem]{Remark}
\newtheorem{ex}[theorem]{Example}

\newcommand{\alg}{\mathrm{alg}}
\newcommand{\bboxplus}{\!\boxplus\!\boxplus}

\title{Analytic subordination for bi-free convolution}
\author{S.T. Belinschi}
\address{CNRS - Institut de Math\'ematiques de Toulouse, 118 Route de Narbonne, 31062 Toulouse,
France}
\email{serban.belinschi@math.univ-toulouse.fr}
\author{H. Bercovici}
\address{Department of Mathematics, Indiana University, 831 E. Third St., Bloomington, IN 47405, USA}
\email{bercovic@indiana.edu}
\author{Y. Gu}
\address{Department of Mathematics and Statistics, Queen's University, Jeffery Hall, 48 University Ave.
Kingston, ON K7L 3N6, Canada}
\email{gu.y@queensu.ca}
\thanks{This work was started during Y. Gu's visit to the IMT, partially supported by 
ANR-11-LABX-0040-CIMI within the program ANR-11-IDEX-0002-02. H. Bercovici was partially supported by a grant from the
National Science Foundation. P. Skoufranis was supported in part by 
Discovery Grant RGPIN-2017-05711 from NSERC (Canada).}
\author{P. Skoufranis}
\address{Department of Mathematics and Statistics
York University, N520 Ross
4700 Keele Street, Toronto, ON M3J 1P3, Canada}
\email{pskoufra@yorku.ca}

\date{}
\begin{document}

\begin{abstract}
In this paper we study some analytic properties of bi-free additive convolution,
both scalar- and operator-valued. We show that using properties of
Voiculescu's subordination functions associated to free additive convolution
of operator-valued distributions, simpler formulas for bi-free convolutions can
be derived. We use these formulas in order to prove several results about atoms of 
bi-free additive convolutions.
\end{abstract}

\maketitle

\section{Introduction}

In his second paper on bi-free independence \cite{BiFree2}, 
Voiculescu provided a linearizing transform for the bi-free
additive convolution of compactly supported probability measures
in the plane $\mathbb R^2$. This formula may be re-written
to naturally involve the subordination functions of the 
free additive convolutions of the two faces (marginals)
of the two probability measures on $\mathbb R^2$ (see
\cite[Remark 2.3]{BiFree2} or Equation \eqref{four} below). Motivated by
the recent work of two of us \cite{GS2}  (see Remark \ref{rmk:biBoolean} below), we note that 
this connection can be more directly justified by
appealing to the operator-valued subordination as
introduced in \cite{V2000}, applied to the restriction
to upper triangular $2\times2$ complex matrices. While 
no freeness over any subalgebra of the  $2\times2$ matrices
appears to be involved in bi-freeness, we nevertheless find
a proof (and slight extension) of the subordination formula \eqref{four} involved in
Voiculescu's methods from \cite{V2000}. Thus, the purpose
of this article is to provide a bi-free equivalent of
Voiculescu's subordination result \cite{V3} and present
a few of its most natural consequences (see Theorem \ref{Main}, 
Proposition \ref{Bound}, and its corollaries below). It turns out that in the case of scalar-valued
distributions, the more general conditionally bi-free additive convolution
introduced by two of us in \cite{GS1}
can also be studied through evaluation of the relevant transforms on upper 
triangular $2\times2$ matrices.

The rest of this paper is organized the following way: in Section \ref{back}
we provide the necessary background in free and bi-free 
probability theories. In Section \ref{sec:bifreesubord} we prove a subordination
relation for bi-free additive convolution of both scalar- and
operator-valued bi-distributions and for bi-free
convolution semigroups. We conclude Section \ref{sec:bifreesubord} with a 
few regularity results that follow from this main result.
In Section \ref{sec:semi} we study analytic properties of the
bi-free convolution semigroups introduced in \cite{GHM}.
Section \ref{sec:c-bi-free} extends some of the results of Section \ref{sec:bifreesubord} to conditionally
bi-free convolution. Finally, the last section, Section \ref{neg}, is dedicated to a discussion of conditional
expectations and traciality in the context of bi-freeness.

\section{Notations and background}\label{back}

\subsection{Noncommutative random variables and (bi)freeness}
For the purposes of this paper, we only need to consider
the definition of the bi-freeness of two pairs of random 
variables. In that, we follow \cite[Section 1.1]{BiFree2},
and refer to \cite{CNS,BiFree1} for full details of the 
analytic and combinatorial aspects of bi-free probability.
Consider a noncommutative probability space $(\mathcal A,\varphi)$,
where $\mathcal A$ is a unital algebra over the field of complex numbers 
$\mathbb C$ and $\varphi\colon\mathcal A\to\mathbb C$ is a unit-preserving
linear functional. We assume that $\mathcal A$ is endowed with an involution 
$*$ and that $\varphi(x^*)=\overline{\varphi(x)}$ for all $x\in\mathcal A$.
We will always assume that $\varphi$ is positive and faithful, meaning 
that $\varphi(x^*x)\ge0$, with equality if and only if $x=0$.
 In this case, we refer to $(\mathcal A, \varphi)$ as a $^*$-noncommutative 
probability space. If, in addition, $\mathcal A$ is a $C^*$-algebra (respectively, 
a $W^*$-algebra) and $\varphi$ is a (normal) state, then $(\mathcal A, \varphi)$ 
is said to be a $C^*$-noncommutative probability space (respectively, a 
$W^*$-noncommutative probability space). Elements of $\mathcal A$ are
called random variables. The distribution of a $k$-tuple of random variables
$(a_1,\ldots,a_k)\in\mathcal A^k$ is by definition the collection of all mixed 
moments of the $k$-tuple:
$$
\mu_{(a_1,\ldots,a_k)}=\{\varphi(a_{i_1}\cdots a_{i_m})\colon m\in\mathbb N,i_1,\dots,i_m\in\{1,\dots,k\}\}.
$$

A two-faced pair of noncommutative random variables in $(\mathcal A, \varphi)$ 
is a pair $(a, b)\in\mathcal A^2$. We consider $a$ as the left random variable and 
$b$ as the right random variable of the pair. A pair $\{(a_1,b_1),(a_2,b_2)\}$ of two-faced
noncommutative random variables in $(\mathcal A,\varphi)$ is said to be bi-free 
if their distribution with respect to $\varphi$ satisfies the
following property: there are two vector spaces $\mathcal{X}_1,
\mathcal{X}_2$ with distinguished state vectors $\xi_1,\xi_2$
(i.e. $\mathcal X_j=\mathbb C\xi_j\oplus\ker\psi_j,$ with
$\psi_j\colon\mathcal X_j\to\mathbb C$ linear, $\psi_j(\xi_j)=1$),
so that if $(\mathcal X,\ker\psi,\xi)=
(\mathcal X_1,\ker\psi_1,\xi_1)*(\mathcal X_2,\ker\psi_2,\xi_2)$,
and $\lambda_j,\rho_j$ are the left and right 
representations of $\mathcal L(\mathcal X_j)$ on $\mathcal L
(\mathcal X)$, respectively, $j=1,2,$ then the joint distribution of
$a_1,a_2,b_1,b_2$ with respect to $\varphi$ in $\mathcal A$ equals 
the joint distribution of variables $\lambda_1(a_1),\lambda_2(a_2),
\rho_1(b_1),\rho_2(b_2)$ with respect to $\varphi_\xi$ in $\mathcal L
(\mathcal X)$. Here $\varphi_\xi(T)=\psi(T(\xi)),T\in\mathcal L
(\mathcal X)$.
 It follows from the definition of bi-freeness that left and right random 
variables of different pairs are classically independent with respect to
 $\varphi$ (see \cite[Proposition 2.16]{BiFree1}) so that, in
particular, they commute.

In \cite{BiFree2}, Voiculescu shows that if the joint 
distribution of $(a_j,b_j)$ is determined by two-bands 
moments (i.e. moments of the form $\varphi(LR)$, where 
$L$ runs through all monomials in the left face and $R$ runs 
through all monomials in the right face), then the same remains 
true for $(a_1+a_2,b_1+b_2)$. If in addition $(\mathcal A,\varphi)$ 
is a $C{}^*$-noncommutative probability space in which $a_j=a_j^*,
b_j=b_j^*$, $a_jb_j=b_ja_j$, then the joint distribution of 
$(a_j,b_j)$ coincides with the moments of a compactly supported
probability measure $\eta_j$ in the plane. We will follow \cite{BiFree2} and 
refer to such variables as {\em bi-partite}. The correspondence
is given via the relation
$$
\varphi(a_j^mb_j^n)=\int_{\mathbb R^2}t^ms^n\,{\rm d}\eta_j(t,s),
\quad m,n\in\mathbb N,j=1,2.
$$
Obviously, under this hypothesis, the distribution of $(a_1+a_2,
b_1+b_2)$ is itself the joint distribution of two commuting
self-adjoint random variables, so that there exists a compactly 
supported probability measure $\eta$ on $\mathbb R^2$ whose moments
coincide with it. The measure $\eta$ depends only on $\eta_1$
and $\eta_2$ via formulae provided, for example, in \cite{CNS}.
The notation $\eta=\eta_1\bboxplus\eta_2$ was introduced
in \cite{BiFree1} and is called the bi-free additive convolution
of $\eta_1$ and $\eta_2$. The measure $\eta$ has the property
that its marginals are the free additive convolutions of the
marginals of $\eta_1$ and $\eta_2$. More specifically, if $\mu_j$ 
is the distribution of $a_j$ and $\nu_j$ is the distribution
of $b_j$, then the first marginal of $\eta$ is $\mu_1\boxplus
\mu_2$ and the second marginal of $\eta$ is $\nu_1\boxplus\nu_2$.

\subsection{Analytic transforms}\label{at}
In order to linearize the bi-free additive convolution, Voiculescu introduced
the partial bi-free $R$-transform, a function of two complex 
variables defined on a neighbourhood of zero in $\mathbb C^2$.
We introduce this function, together with its single-variable
analogue, and indicate how it allows one to interpret the 
operation $\bboxplus$ in terms of the single-variable 
analytic subordination functions \cite{V3,Biane1,BBSubord}.

First, define
\begin{eqnarray*}
G_{\eta_j}(z,w)=\varphi\left((z-a_j)^{-1}(w-b_j)^{-1}\right)=
\int_{\mathbb R^2}\frac{{\rm d}\eta_j(t,s)}{(z-t)(w-s)},
\quad\quad \\
G_{\mu_j}(z)=\varphi\left((z-a_j)^{-1}\right)=\int_\mathbb R\frac{{\rm 
d}\mu_j(t)}{z-t},\quad
G_{\nu_j}(w)=\varphi\left((w-b_j)^{-1}\right)=\int_\mathbb R\frac{{\rm 
d}\nu_j(t)}{w-s},
\end{eqnarray*}
for $z\in\mathbb C\setminus\sigma(a_j),w\in\mathbb C\setminus
\sigma(b_j)$, $j=1,2$. Here $\sigma(T)$ denotes the spectrum of the
operator $T$. We shall refer to all three of these functions
as the Cauchy transforms of the corresponding measures. 
Observe that they determine uniquely the probability measures in question,
and depend only on the distribution of $(a_j,b_j)$ with respect to $\varphi$. 
Nevertheless, in the following, we will sometimes write $G_{(a_j,b_j)}$ for 
$G_{\eta_j}$, or $G_{a_j}$ for $G_{\mu_j}$ (respectively, $G_{b_j}$ for $G_{\nu_j}$).

We remind the reader of some properties of Cauchy transforms of positive measures. 
Let us start with the simpler one-variable Cauchy transform: it is an analytic function 
$G_\mu$ sending the upper half-plane $\mathbb C^+$ of the complex plane into the 
lower half-plane $\mathbb C^-$, $G_\mu(\overline{z})=\overline{G_\mu(z)}$, 
and we have $\mu(\mathbb R)=\lim_{y\to+\infty}iyG_\mu(iy).$
The topological support of the measure $\mu$, denoted by $\text{supp}(\mu)$, is 
characterized by the fact that $G_\mu$ extends analytically with real values to its 
complement. One easily sees that $G_\mu$ is decreasing on each connected 
component of $\mathbb R\setminus\text{supp}(\mu)$. It is negative on
$(-\infty,\inf\text{supp}(\mu))$ and positive on $(\sup\text{supp}(\mu),+\infty)$.
However, $G_\mu$ may pass through zero on a bounded component of
$\mathbb R\setminus\text{supp}(\mu)$. If $\mu$ is a probability measure
(that is, $\mu(\mathbb R)=1$), then a simple geometric argument 
shows that 
\begin{eqnarray*}
\lefteqn{G_\mu(z)\in\left \{u+iv\in\mathbb C\colon u^2+\left(v+\frac1{2\Im z}\right)^2\le\frac{1}{(2\Im z)^2},\right.}\\
& & \left.
v\leq-\min\left\{\frac{\Im z}{(\Re z-m)^2+(\Im z)^2},\frac{\Im z}{(\Re z-M)^2+(\Im z)^2}\right\}\right\}
\end{eqnarray*} 
whenever
$\text{supp}(\mu)\subseteq[m,M]$. Indeed, more precisely, $G_\mu(z)$ is a limit of convex combinations
of points $(z-t)^{-1}$, $m\le t\le M$, and these are points which lay on the arc of the circle centered at 
$-1/2\Im z$ and of radius $1/2\Im z$ which is bordered by the points $(z-m)^{-1}$ and $(z-M)^{-1}$
and does not contain zero.

Unfortunately, geometric properties of the two-variable Cauchy transform are nowhere near as 
nice as those of the one-variable Cauchy transform. However, given a compactly supported 
Borel probability measure $\eta$ on $\mathbb R^2$, one can still deduce some useful properties 
of $G_\eta(z,w)$. It is quite obvious that $G_\eta(\overline{z},\overline{w})=\overline{G_\eta(z,w)}$,
and that $G_\eta$ is analytic as a function of two complex variables on the set 
$\{(z,w)\in\mathbb C^2\colon(\{z\}\times\mathbb R)\cap\text{supp}(\eta)
=(\mathbb R\times\{w\})\cap\text{supp}(\eta)=\varnothing\}$. That is to say, 
if $\mu,\nu$ are the marginals of $\eta$, then the only part of $\mathbb C^2$
on which $G_\eta$ may not be analytic is the union of two strips, 
namely $\text{supp}(\mu)\times\mathbb C$ and $\mathbb C\times\text{supp}(\nu)$.
In particular, while the domain of analyticity of $G_\eta$ may not be simply connected, 
it is a connected open subset of $\mathbb C^2$ whose complement is a closed set of
(Hausdorff) dimension at most 3.
Regrettably, $G_\eta$ does not preserve half-planes, and may map elements from
$\mathbb C^+\times\mathbb C^+$ in $\mathbb R$, including possibly zero. The 
zero set of a nonconstant two-variable analytic function is an analytic set which
is either empty or of complex dimension one (we refer to \cite{Chirka} for definition
and properties of analytic sets). Specifically, if there exists a point $(z_0,w_0)$
in the domain of analyticity of $G_\eta$ such that $G_\eta(z_0,w_0)=0$, and
the map $z\mapsto G_\eta(z,w_0)$ has a finite number of zeros in a given
bounded neighborhood of $(z_0,w_0),$ then, by Weiestrass' preparation 
theorem, there exist a nonempty open bidisk $U$ centered at $(z_0,w_0)$, 
an integer $k\in\mathbb N$ and one-variable analytic functions $c_1,
\dots,c_k$ defined on the first coordinate of $U$ such that
 $G_\eta(z,w)=((z-z_0)^k+c_1(w)(z-z_0)^{k-1}+\cdots+c_k(w))\phi(z,w)$,
where $\phi(z,w)$ is a zero-free analytic function on $U$.
For $k=1$ we recover the classical analytic implicit function theorem:
$z=z_0-c_1(w)$ is the implicit function. Thus, the zero set of $G_\eta$
cannot be compactly contained in its domain of analyticity. We would like 
to make this statement more precise for the case in which both $z_0$ and
$w_0$ belong to a half-plane (upper or lower - not necessarily the same).
Say $z_0\in H_1,w_0\in H_2$, $H_j\in\{\mathbb C^\pm\}$. In this case it 
is clear that neither $z\mapsto G_\eta(z,w_0)$ nor $w\mapsto 
G_\eta(z_0,w)$ is constantly equal to zero,
so Weierstrass' preparation theorem applies to it for both coordinates. 
Consider the restriction of $G_\eta$ to $H_1\times H_2$. By 
\cite[Definition 2.1.2]{Chirka}, $\{(z,w)\in H_1\times H_2\colon G_\eta(z,w)=0\}$
is an analytic subset of $H_1\times H_2$. Let $Z$ be a connected component of
this set. This is a principal analytic set (the zero set of a two-variable analytic 
function) of dimension and codimension one, so, as \cite[Theorem 2.3 and 
Corollary 2.8.2]{Chirka} inform us, its singular points form a discrete set (a 
point of $Z$ is regular if it has a neighbourhood $U$ such that $U\cap Z$ is 
a manifold, and it is singular if it is not regular). Of course, these singular points 
may very well accumulate near the boundary of the natural domain of $G_\eta$.
Tautologically, $Z$ is also irreducible, so that the intersection of $Z$ with any other
one-dimensional analytic subset of $H_1\times H_2$ is either a discrete (possibly
empty) set, or contains $Z$ (see \cite[Sections 5.3-5.6]{Chirka}).
This observation will be useful later on when we need to compare zero sets
of different two-variable Cauchy transforms.

It is known from \cite{V-JFA} that if one defines
$K_{\mu_j}(z)$ as the inverse of $G_{\mu_j}(z)$ on a 
neighbourhood of infinity (so that $K_{\mu_j}(0)=\infty$), then
the function $R_{\mu_j}(z)=K_{\mu_j}(z)-\frac1z$, called the (free) {\em $R$-transform} of $\mu_j$, is analytic (instead
of just meromorphic) on the same neighbourhood of zero and
satisfies the relation $R_{\mu_1\boxplus\mu_2}(z)=R_{\mu_1}(z)
+R_{\mu_2}(z)$ for $z$ in a small enough neighbourhood of zero.
Observe that if we define $\omega_{a_1}(z)=K_{\mu_1}(G_{\mu_1
\boxplus\mu_2}(z))$ and $\omega_{a_2}(z)=K_{\mu_2}(G_{\mu_1
\boxplus\mu_2}(z))$, then the relation satisfied by the 
$R$-transforms can be re-written as
\begin{equation}\label{subord}
\omega_{a_1}(z)+\omega_{a_2}(z)-z=\frac{1}{G_{\mu_1
\boxplus\mu_2}(z)}=\frac{1}{G_{\mu_1}(\omega_{a_1}(z))}
=\frac{1}{G_{\mu_2}(\omega_{a_2}(z))}.
\end{equation}
It has been shown that the functions $\omega_{a_j}$, $j=1,2$, called
the {\em subordination functions}, extend analytically as self-maps of 
the complex upper half-plane $\mathbb C^+$ and Equation \eqref{subord} 
holds for all $z\in\mathbb C^+$ (see \cite{V3,Biane1,BBSubord}). 
Of course, a similar relation holds for $R_{\nu_j}(w),$
$G_{\nu_j}(w)$ and $\omega_{b_j}(w)$, $j=1,2$.

For the measure $\eta_j$ (which is the distribution of
the pair $(a_j,b_j)$ with respect to $\varphi$), Voiculescu introduces
in \cite[Theorem 2.1]{BiFree2} the function 
\begin{equation}\label{BiR}
R_{(a_j,b_j)}(z,w)=R_{\eta_j}(z,w)=1+zR_{\mu_j}(z)+wR_{\nu_j}(w)-
\frac{zw}{G_{\eta_j}(K_{\mu_j}(z),K_{\nu_j}(w))},
\end{equation}
for $z,w$ in a small enough bi-disk centred at zero (also see \cite[Section 7.2]{S2}). 
This function is called  the {\em partial bi-free $R$-transform} of $(a_j, b_j)$ (or of $\eta_j$).
Observe first that this function is indeed well-defined, 
including at zero, since 
$$
\lim_{w\to0}\lim_{z\to0}
\frac{G_{\eta_j}(K_{\mu_j}(z),K_{\nu_j}(w))}{zw}=
\lim_{w\to0}\frac{G_{\nu_j}(K_{\nu_j}(w))}{w}=1.
$$
The limits can clearly be permuted. In particular, $R_{\eta_j}
(0,0)=0$. Theorem 2.1 combined with Section 1.2 from \cite{BiFree2}
provide the following:
\begin{equation}\label{BiLin}
R_{\eta_1}(z,w)+R_{\eta_2}(z,w)=R_{\eta_1\bboxplus\eta_2}(z,w),
\quad |z|+|w|\text{ sufficiently small}.
\end{equation}
Given the linearizing property of the one-variable
$R$-transform, this is equivalent to
$$
\frac{zw}{G_{\eta_1}(K_{\mu_1}(z),K_{\nu_1}(w))}+
\frac{zw}{G_{\eta_2}(K_{\mu_2}(z),K_{\nu_2}(w))}-1\!=\!
\frac{zw}{G_{\eta_1\bboxplus\eta_2}(K_{\mu_1\boxplus\mu_2}(z),
K_{\nu_1\boxplus\nu_2}(w))}.
$$
This relation and Equation \eqref{subord} allow us to write a formula for 
$G_{\eta_1\bboxplus\eta_2}$ defined on all of 
$(\mathbb C\setminus\sigma(a_1+a_2))\times
(\mathbb C\setminus\sigma(b_1+b_2))$ involving the subordination 
functions of free additive convolution.
If we divide by $zw$ and replace in the above $z$ by 
$G_{\mu_1\boxplus\mu_2}(z)$ and $w$ by $G_{\nu_1\boxplus\mu_2}(w)$, 
then
\begin{eqnarray}\label{four}
\lefteqn{\frac{1}{G_{\eta_1}(\omega_{a_1}(z),\omega_{b_1}(w))}+
\frac{1}{G_{\eta_2}(\omega_{a_2}(z),\omega_{b_2}(w))} }\quad \quad
\quad \quad\quad \quad\quad \quad\\
& = & \frac{1}{G_{\mu_1\boxplus\mu_2}(z)G_{\nu_1\boxplus\nu_2}(w)}
+\frac{1}{G_{\eta_1\bboxplus\eta_2}(z,w)}.
\nonumber
\end{eqnarray}
This relation clearly holds for $|z|$ and $|w|$ sufficiently large as an
equality of analytic functions. If rewritten as 
\begin{eqnarray}
\lefteqn{G_{\eta_1\bboxplus\eta_2}(z,w)\left(G_{\eta_1}(\omega_{a_1}(z),\omega_{b_1}(w))+
G_{\eta_2}(\omega_{a_2}(z),\omega_{b_2}(w))\right) }\label{univ}\\
& = & G_{\eta_1}(\omega_{a_1}(z),\omega_{b_1}(w))\left(\frac{G_{\eta_1\bboxplus\eta_2}(z,w)}{G_{\mu_1\boxplus\mu_2}(z)G_{\nu_1\boxplus\nu_2}(w)}
+1\right)G_{\eta_2}(\omega_{a_2}(z),\omega_{b_2}(w)),\nonumber
\end{eqnarray}
then it holds for any $z\in\mathbb C\setminus\sigma(a_1+a_2),w\in\mathbb C
\setminus\sigma(b_1+b_2)$, as an equality of meromorphic functions 
(see \cite[Remark 2.3 and Lemma 2.6]{BiFree2}
). In fact, the only poles may come from zeros of $G_{\mu_1\boxplus\mu_2}$ and of $G_{\nu_1\boxplus\nu_2}$.
These functions can have zeros only in $\text{co}(\sigma(a_1+a_2))\setminus\sigma(a_1+a_2)$ (respectively
$\text{co}(\sigma(b_1+b_2))\setminus\sigma(b_1+b_2)$ - we have denoted by $\text{co}(A)$
the convex hull of the set $A$). However, Equation \eqref{subord} guarantees that 
exactly one of $\omega_{a_1},\omega_{a_2}$  has a simple pole at the 
zero of $G_{\mu_1\boxplus\mu_2}$ (with a similar statement 
for $b$ and $\nu$). The behavior of the two-variable Cauchy
transform at infinity guarantees that the right-hand side of 
the above equality is actually analytic in such a point, and
thus the equality is an equality of analytic functions on 
$(\mathbb C\setminus\sigma(a_1+a_2))\times(\mathbb C
\setminus\sigma(b_1+b_2))$.

We will see below that in fact Equation \eqref{four} extends to
$(\mathbb C^+\times\mathbb C^+)\cup(\mathbb C^-\times\mathbb C^-)$ as an equality 
of meromorphic functions, and  $\{(z,w)\in\mathbb C^+\times\mathbb C^+
\colon G_{\eta_1\bboxplus\eta_2}(z,w)=0\}\supseteq\{(z,w)\in\mathbb C^+\times\mathbb C^+
\colon G_{\eta_j}(\omega_{a_j}(z),\omega_{b_j}(w))=0\}$, $j=1,2$.

\subsection{Operator-valued random variables, their analytic transforms, 
and (bi)freeness with amalgamation}
Most importantly for us, Voiculescu extended in \cite{V2000} 
the analytic subordination results  from above to self-adjoint 
random variables which are {\em free with amalgamation} 
over some subalgebra. We outline his result below.



Let $(M, E,B)$ be an operator-valued $W^*$-noncommutative 
probability space; that is, $B\subseteq M$ is a unital inclusion
of (unital) $W^*$-algebras, and $E\colon M\to B$ is a 
unit-preserving conditional expectation. Let $X_1=X_1^*,X_2=X_2^*
\in M$ be free over $B$ with respect to $E$. Then there exists a 
countable family $\omega=\{\omega_n\}_{n\in\mathbb N}$, with each
$\omega_n$ defined on a subset of $M_n(B)$, such that
$$
(E\otimes{\rm Id}_{M_n(B)})\left[(v-(X_1+X_2)\otimes I_n)^{-1}\right]=
(E\otimes{\rm Id}_{M_n(B)})\left[(\omega_n(v)-X_1\otimes I_n)^{-1}\right],
$$
for all $v\in M_n(B)$ with strictly positive imaginary part or of
inverse of sufficiently small norm. The functions $\omega_n$ increase
the imaginary part of $v$ if $\Im v>0$, and the dependence on $n$
satisfies certain compatibility conditions (see \cite{V2000,V1}). As it will 
usually be clear from the context, from now on we supress the level
$n$ from our notation. 



The functions of the type 
$$
G_X=\{G_{X,n}\}_{n\in\mathbb N},\quad G_{X,n}(v)=(E\otimes{\rm Id}_{M_n(B)})\left[
(v-X\otimes I_n)^{-1}\right]
$$ 
are natural extensions of the classical Cauchy transforms and share 
many of the properties of their classical, complex-valued counterparts. Voiculescu 
showed in \cite{V*} that they allow the definition of 
$B$-valued $R$-transforms via the exact same procedure 
as for the complex-valued $R$-transforms (for a discussion of the natural domain of $K_X$,
see \cite{BMS}), and that these 
$R$-transforms satisfy $R_{X_1+X_2,n}(v)=R_{X_1,n}(v)+R_{X_2,n}(v)$,
$n\in\mathbb N$, on a small neighbourhood of zero in $M_n(B)$, for $X_1,X_2$ 
free with respect to $E$ (see \cite{V2}). As for the functions $\omega$ above,
in the following we will suppress the index $n$ from the notations of $G$ and $R$
whenever the space on which they are defined is clear from the context. 

An argument
similar to the one used to prove \eqref{subord} shows that
the $B$-valued subordination functions satisfy precisely
the same equation \eqref{subord}, but with variables $v\in
M_n(B),\Im v>0$ instead of variables $z\in\mathbb C^+$
(see \cite{BMS} for details). 

An operator-valued version of the analytic transforms of bi-freeness 
has been elaborated by one of us in \cite{S}. We present a version of
Equation \eqref{four} for operator-valued transforms. We 
consider a $C^*$-$B$-$B$-noncommutative probability space; the case when
$B$ is finite dimensional is of a special interest to us (see 
\cite[Definitions 2.5 and 5.1]{S} for details).  An important 
difference from the case of $\mathbb C$-valued case 
comes from the fact that a noncommutative algebra
may receive a natural ``opposite'' structure. Thus, 
for instance, if analytic transforms of left random variables 
in a $B$-$B$-noncommutative probability space coincide with 
Voiculescu's analytic transforms introduced above, analytic transforms 
of the right random variables, while defined the same way (and thus 
having the same analytic properties), are viewed as being defined on 
(open subsets of) $B^{\rm op}$, the algebra having the same underlying set 
and vector space structure as $B$, but with multiplication defined by
$b\cdot_{\rm op}b'=b'b$. Consequently, a ${C}^*$-, or a ${W}^*$-$B$-$B$
noncommutative probability space requires, beyond the data $(M,E,B)$
described above, a way to view $B$ and $B^{\rm op}$ simultaneously as
subalgebras of $M$, that is, a linear homomorphism $\varepsilon\colon B\otimes B^{\rm op}\to M$.
This homomorphism satisfies certain conditions for which we refer to 
\cite[Definition 2.5]{S} (see also \cite{CNS}). For our study, it is important to note that we must add an 
$\ell$ or an $r$ to each analytic transform defined on $B$, corresponding to whether it is viewed as 
being defined on $B$ or $B^{\rm op}$. In 
addition, the $B$-$B$-valued equivalent of $G_\eta(z,w)$ (or, more 
precisely, of $G_\eta(z^{-1},w^{-1})/zw$), becomes a function of three
variables $M_{(X,Y)}(b,c,d)$, linear in $c\in B$, and for which
$b\in B$ is a ``left'' indeterminate and $d\in B^{\rm op}$ is a 
``right'' indeterminate (whether $c$ is viewed as a left or right 
indeterminate is irrelevant). More specifically, for a 
bi-random variable $(X,Y)$ we define the $B$-valued partial moment generating
function
$$
M_{(X,Y)}(b,c,d):=\sum_{n,m\ge0}E(({\rm L}_bX)^n({\rm R}_dY)^m{\rm R}_c),\quad
b,c,d\in B,\|b\|,\|d\|\text{ small}.
$$
The moment generating functions of the left and right variables are
$M_X^\ell(b)=\sum_{n\ge0}E(({\rm L}_bX)^n)$ and $M_X^r(d)=\sum_{n\ge0}
E(({\rm R}_dX)^n)$, respectively.
For the purposes of this paper, the reader is invited to see ${\rm L}_b$ and 
${\rm R}_d$ just as special ways of viewing the scalar algebras $B$ and 
$B^{\rm op}$ embedded in the noncommutative algebra $M$.
Thus, Voiculescu's subordination relations from above are
re-written in terms of the two moment generating functions as 
$$
G_{X_1+X_2}(b^{-1})=M_{X_1+X_2}^\ell(b)b=
M_{X_1}^\ell(\omega(b^{-1})^{-1})\omega(b^{-1})^{-1},
$$ 
$$
G_{Y_1+Y_2}(d^{-1})=dM_{Y_1+Y_2}^r(d)=\omega(d^{-1})^{-1}
M_{Y_1}^r(\omega(d^{-1})^{-1}),
$$ 
respectively (the convention from \cite{S} is slightly different 
from ours: the $G(b)$ from \cite{S} is $G(b^{-1})$ here).

For our purposes, we prefer to view $M_{(X,Y)}$  as an 
analytic function from $B\times B$ with values in $\mathcal L(B)$,
the space of continuous linear operators from $B$ to itself. Viewed
as such, we have $M_{(X,Y)}(0,c,0)=c$, i.e. $M_{(X,Y)}(0,\cdot,0)=
{\rm Id}_B$. Since the correspondence $B\times B\ni(b,d)\mapsto 
M_{(X,Y)}(b,\cdot,d)\in\mathcal L(B)$ is analytic, we conclude
that on a small enough norm-neighbourhood of $(0,0)$, the element
$M_{(X,Y)}(b,\cdot,d)\in\mathcal L(B)$ is invertible as a linear map
from $B$ to $B$. We define $\Psi_{(X,Y)}(b,\cdot,d)=b^{-1}M_{(X,Y)}
(b,\cdot,d)^{\langle-1\rangle}d^{-1}\in\mathcal L(B)$, where 
$M_{(X,Y)}(b,\cdot,d)^{\langle-1\rangle}$ is the inverse of 
$M_{(X,Y)}(b,\cdot,d)$ in $\mathcal L(B)$.

With these notations, the partial $R$-transform of $(X,Y)$ defined in 
\cite[Section 5]{S} is the analytic map
$B\times B\ni(b,d)\mapsto R_{(X,Y)}(b,\cdot,d)\in\mathcal L(B)$
uniquely determined on a neighbourhood of $(0,0)$ by the initial
condition $R_{(X,Y)}(0,\cdot,0)=0$ (the zero element in 
$\mathcal L(B)$), and the functional equation 
$$
R_{(X,Y)}(M^\ell_X(b)b,c,dM^r_Y(d))=M^\ell_X(b)c+cM^r_Y(d)
-M^\ell_X(b)b\Psi_{(X,Y)}(b,c,d)dM^r_Y(d)-c.
$$
If $(X_1,Y_1)$ and $(X_2,Y_2)$ are bi-free over $B$ with respect to $E$, then the partial $R$-transform satisfies 
$$
R_{(X_1+X_2,Y_1+Y_2)}(b,c,d)=R_{(X_1,Y_1)}(b,c,d)+R_{(X_2,Y_2)}(b,c,d).
$$ 
Using the functional equation defining $R_{(X,Y)}$, we obtain
\begin{eqnarray*}
\lefteqn{M^\ell_{X_1+X_2}(b)c+cM^r_{Y_1+Y_2}(d)-M^\ell_{X_1+X_2}(b)
b\Psi_{(X_1+X_2,Y_1+Y_2)}(b,c,d)dM^r_{Y_1+Y_2}(d)-c}\\
&= & R_{(X_1+X_2,Y_1+Y_2)}(M^\ell_{X_1+X_2}(b)b,c,dM^r_{Y_1+Y_2}(d))\\
& = & R_{(X_1,Y_1)}(M^\ell_{X_1+X_2}(b)b,c,dM^r_{Y_1+Y_2}(d))+
R_{(X_2,Y_2)}(M^\ell_{X_1+X_2}(b)b,c,dM^r_{Y_1+Y_2}(d)).
\end{eqnarray*}
The subordination relation provides
\begin{eqnarray*}
\lefteqn{R_{(X_j,Y_j)}(M^\ell_{X_1+X_2}(b)b,c,dM^r_{Y_1+Y_2}(d))}\\
& = & R_{(X_j,Y_j)}(M^\ell_{X_j}(\omega_{X_j}(b^{-1})^{-1})
\omega_{X_j}(b^{-1})^{-1},c,\omega_{Y_j}(d^{-1})^{-1}
M_{Y_j}^r(\omega_{Y_j}(d^{-1})^{-1}))\\
& = & M^\ell_{X_j}(\omega_{X_j}(b^{-1})^{-1})c+c
M_{Y_j}^r(\omega_{Y_j}(d^{-1})^{-1})-c\\
& & \mbox{}-M^\ell_{X_j}(\omega_{X_j}(b^{-1})^{-1})
\omega_{X_j}(b^{-1})^{-1}\\
& & \mbox{}\times\Psi_{(X_j,Y_j)}
\left(\omega_{X_j}(b^{-1})^{-1},c,\omega_{Y_j}(d^{-1})^{-1}\right)
\omega_{Y_j}(d^{-1})^{-1}M_{Y_j}^r(\omega_{Y_j}(d^{-1})^{-1}).
\end{eqnarray*}
This provides the following operator-valued analogue of relation \eqref{four}:
\begin{eqnarray}
\lefteqn{M^\ell_{X_1+X_2}(b)c+cM^r_{Y_1+Y_2}(d)-M^\ell_{X_1+X_2}(b)b
\Psi_{(X_1+X_2,Y_1+Y_2)}(b,c,d)dM^r_{Y_1+Y_2}(d)-c}\nonumber\\
& = & M^\ell_{X_1}(\omega_{X_1}(b^{-1})^{-1})c+c
M_{Y_1}^r(\omega_{Y_1}(d^{-1})^{-1})-c\nonumber\\
& & \mbox{}-M^\ell_{X_1}(\omega_{X_1}(b^{-1})^{-1})
\omega_{X_1}(b^{-1})^{-1}\nonumber\\
& & \mbox{}\times\Psi_{(X_1,Y_1)}
\left(\omega_{X_1}(b^{-1})^{-1},c,\omega_{Y_1}(d^{-1})^{-1}\right)
\omega_{Y_1}(d^{-1})^{-1}M_{Y_1}^r(\omega_{Y_1}(d^{-1})^{-1})\nonumber\\
& & \mbox{}+ M^\ell_{X_2}(\omega_{X_2}(b^{-1})^{-1})c+c
M_{Y_2}^r(\omega_{Y_2}(d^{-1})^{-1})-c\nonumber\\
& & \mbox{}-M^\ell_{X_2}(\omega_{X_2}(b^{-1})^{-1})
\omega_{X_2}(b^{-1})^{-1}\nonumber\\
& & \mbox{}\times\Psi_{(X_2,Y_2)}
\left(\omega_{X_2}(b^{-1})^{-1},c,\omega_{Y_2}(d^{-1})^{-1}\right)
\omega_{Y_2}(d^{-1})^{-1}M_{Y_2}^r(\omega_{Y_2}(d^{-1})^{-1}).\label{five}
\end{eqnarray}
Furthermore, recalling that $G_X(b)=M_X^\ell(b^{-1})b^{-1}$, the 
$B$-valued version of \eqref{subord} written under the
form 
$$
G_{X_1+X_2}(b^{-1})b^{-1}=G_{X_1}(\omega_{X_1}(b^{-1}))
\omega_{X_1}(b^{-1})+G_{X_2}(\omega_{X_2}(b^{-1}))
\omega_{X_2}(b^{-1})-1
$$ 
allows us to simplify the above to 
\begin{eqnarray}
\lefteqn{M^\ell_{X_1+X_2}(b)b
\Psi_{(X_1+X_2,Y_1+Y_2)}(b,c,d)dM^r_{Y_1+Y_2}(d)+c}\nonumber\\
& = & \mbox{}M^\ell_{X_1}(\omega_{X_1}(b^{-1})^{-1})
\omega_{X_1}(b^{-1})^{-1}\nonumber\\
& & \mbox{}\times\Psi_{(X_1,Y_1)}
\left(\omega_{X_1}(b^{-1})^{-1},c,\omega_{Y_1}(d^{-1})^{-1}\right)
\omega_{Y_1}(d^{-1})^{-1}M_{Y_1}^r(\omega_{Y_1}(d^{-1})^{-1})\nonumber\\
& & \mbox{}+M^\ell_{X_2}(\omega_{X_2}(b^{-1})^{-1})
\omega_{X_2}(b^{-1})^{-1}\nonumber\\
& & \mbox{}\times\Psi_{(X_2,Y_2)}
\left(\omega_{X_2}(b^{-1})^{-1},c,\omega_{Y_2}(d^{-1})^{-1}\right)
\omega_{Y_2}(d^{-1})^{-1}M_{Y_2}^r(\omega_{Y_2}(d^{-1})^{-1}).\label{six}
\end{eqnarray}
Since we usually prefer to deal with $G$ rather than $M$, we perform
in the above the changes of variable $b\mapsto b^{-1}$ and $d\mapsto
d^{-1}$ in order to obtain
\begin{eqnarray}
\lefteqn{G_{X_1+X_2}(b)
\Psi_{(X_1+X_2,Y_1+Y_2)}(b^{-1},c,d^{-1})G_{Y_1+Y_2}(d)+c}\nonumber\\
& = & \mbox{}G_{X_1}(\omega_{X_1}(b))\Psi_{(X_1,Y_1)}
\left(\omega_{X_1}(b)^{-1},c,\omega_{Y_1}(d)^{-1}\right)
G_{Y_1}(\omega_{Y_1}(d))\nonumber\\
& & \mbox{}+G_{X_2}(\omega_{X_2}(b))\Psi_{(X_2,Y_2)}
\left(\omega_{X_2}(b)^{-1},c,\omega_{Y_2}(d)^{-1}\right)
G_{Y_2}(\omega_{Y_2}(d)),\label{seven}
\end{eqnarray}
a precise analogue of Equation \eqref{four} (the parallel is made 
more obvious by the fact that $\Psi(b^{-1},\cdot,d^{-1})=
G(b,\cdot,d)^{\langle-1\rangle}$, as $G(b^{-1},b^{-1}cd^{-1},d^{-1}):=
M(b,c,d)$). This relation holds 
for all $b,d$ for which the functions $\Psi$ are defined
as analytic extensions from the set of elements 
$b,d$ satisfying the requirement that $\|b^{-1}\|$
and $\|d^{-1}\|$ are sufficiently small. This is an open set, so
that the above relation does determine $\Psi_{(X_1+X_2,Y_1+Y_2)}$
uniquely from the knowledge of $\Psi_{(X_j,Y_j)},j=1,2,$ and of the
free additive convolution of operator-valued distributions. 
We record next a slightly different version of \eqref{seven}, 
which resembles the scalar reduced partial
$R$-transform. Its validity follows trivially from \eqref{seven}:
\begin{eqnarray}
\lefteqn{\Psi_{(X_1+X_2,Y_1+Y_2)}(K_{X_1+X_2}(b)^{-1},c,
K_{Y_1+Y_2}(d)^{-1})-b^{-1}cd^{-1}=}\nonumber\\
& & \mbox{}\Psi_{(X_1,Y_1)}
\left(K_{X_1}(b)^{-1},c,K_{Y_1}(d)^{-1}\right)-b^{-1}cd^{-1}\nonumber\\
& & \mbox{}+\Psi_{(X_2,Y_2)}
\left(K_{X_2}(b)^{-1},c,K_{Y_2}(d)^{-1}\right)-b^{-1}cd^{-1},
\label{eight}
\end{eqnarray}

The transforms $G,M,\Psi$ and $R$ do not fully characterize the
joint distribution of $(X,Y)$, but only the ``band moments.'' However, 
this will suffice for our main purpose of studying general distributions 
of bi-partite bi-free random variables (we remind the reader that by ``bi-partite'' we mean 
that all left random variables commute with all right random variables). In this context, our interest is
mainly in bi-freeness with amalgamation over 
finite dimensional algebras. The construction providing an 
$M_n(\mathbb C)$-$M_n(\mathbb C)$-noncommutative probability space from
a classical noncommutative probability space $(\mathcal A,\varphi)$
is the following (see \cite[Section 6]{S2} or \cite[Section 4]{S}).
First, one defines the left and right actions
$$
{\rm L}_b(T)=\left[\sum_{k=1}^nb_{ik}T_{kj}\right]_{i,j=1}^n,\quad
{\rm R}_d(T)=\left[\sum_{k=1}^nd_{kj}T_{ik}\right]_{i,j=1}^n,
$$
for all $b,d\in M_n(\mathbb C),T\in
M_n(\mathcal A).$ Both 
${\rm L}_b$ and ${\rm R}_d$ are indeed bounded linear maps on $M_n(\mathcal A)$.
The correspondences $b\mapsto {\rm L}_b$ and $d\mapsto{\rm R}_d$ are 
algebra homomorphisms from $M_n(\mathbb C)$ and
$M_n(\mathbb C)^{\rm op}$, respectively, into 
$\mathcal L(M_n(\mathcal A)).$
We define the right algebra 
$\mathcal L(M_n(\mathcal A))_r$ as the set 
$$
\{Z\in\mathcal L(M_n(\mathcal A))\colon
Z{\rm L}_b={\rm L}_bZ\text{ for all }b\in M_n(\mathbb C)\},
$$
and $\mathcal L(M_n(\mathcal A))_\ell$ the same way,
but with ${\rm L}$ replaced by ${\rm R}$. One embeds $M_n(\mathcal A)$
in $\mathcal L(M_n(\mathcal A))_r$ via
$$
Z\mapsto{\rm R}(Z), \quad {\rm R}(Z)(T)=\left[\sum_{k=1}^nZ_{kj}T_{ik}
\right]_{i,j=1}^n\text{ for all }T\in M_n(\mathcal A),
$$
and in $\mathcal L(M_n(\mathcal A))_\ell$ via
$$
Z\mapsto{\rm L}(Z),\quad{\rm L}(Z)(T)=\left[\sum_{k=1}^nZ_{ik}T_{kj}
\right]_{i,j=1}^n\text{ for all }T\in M_n(\mathcal A).
$$
The map $Z\mapsto{\rm L}(Z)$ is an injective algebra $*$-homomorphism from
$M_n(\mathcal A)$ into $\mathcal L(M_n(\mathcal A))$,
and $Z\mapsto{\rm R}(Z)$ from $M_n(\mathcal A^{\rm op})^{\rm op}$ into
$\mathcal L(M_n(\mathcal A))$.
The conditional expectation $E_n\colon
\mathcal L(M_n(\mathcal A))\to M_n(\mathbb C)$
defined by $E_n[W]=\left[\varphi(W(I_n)_{ij})\right]_{i,j=1}^n$
satisfies $E_n[{\rm L}(Z)]=\left[\varphi(Z_{ij})\right]_{i,j=1}^n$ and
$E_n[{\rm R}(Z)]=\left[\varphi(Z_{ij})\right]_{i,j=1}^n$. That is, the 
distribution of $Z\in M_n(\mathcal A)$ with respect to 
$\varphi\otimes{\rm Id}_{M_n(\mathbb C)}$ is the same as 
the distribution of ${\rm L}(Z)$ (respectively, ${\rm R}(Z)$) with respect to $E_n$.
We note that ${\rm R}_d{\rm R}(Y)={\rm R}(Yd)$, ${\rm L}_b{\rm L}(X)={\rm L}(bX)$, 
and ${\rm R}_{I_n}={\rm L}_{I_n}={\rm Id}_{M_n(\mathcal A)}$.
Then the map $M_{(X,Y)}$ defined above is written as 
\begin{eqnarray*}
M_{(X,Y)}(b,c,d) & = & E_n\left[(1-{\rm L}_b{\rm L}(X))^{-1}(1-{\rm R}_d{\rm R}(Y))^{-1}{\rm R}_c
\right]\\
& = & (\varphi\otimes{\rm Id}_{M_n(\mathbb C)})
\left[(1-{\rm L}(bX))^{-1}(1-{\rm R}(Yd))^{-1}{\rm R}(c) (I_n)\right]\\
& = & (\varphi\otimes{\rm Id}_{M_n(\mathbb C)})
\left[{\rm L}((1-bX)^{-1}){\rm R}((1-Yd)^{-1}){\rm R}(c) (I_n)\right]\\
& = & (\varphi\otimes{\rm Id}_{M_n(\mathbb C)})
\left[{\rm L}((1-bX)^{-1}){\rm R}(c(1-Yd)^{-1}) (I_n)\right]\\
&  = & (\varphi\otimes{\rm Id}_{M_n(\mathbb C)})
\left[{\rm L}((1-bX)^{-1})(c(1-Yd)^{-1})\right]\\
& = &  (\varphi\otimes{\rm Id}_{M_n(\mathbb C)})
\left[(1-bX)^{-1}c(1-Yd)^{-1}\right].
\end{eqnarray*}
Thus, vitally for us, the band-moment generating function
for a pair of faces $({\rm L}(X),{\rm R}(Y))$ coincides with the band moment 
generating function of $(X,Y)$. This allows us to extend 
the map $M_{(X,Y)}(b,\cdot,d)$, as its scalar-valued analogue, 
to $\{b\in M_n(\mathbb C)\colon\pm\Im b>0\}\times
\{d\in M_n(\mathbb C)^{\rm op}\colon\pm\Im d>0\}$.

To conclude this section, we state a particular case of \cite[Theorem 6.3.1]{S2} (alternatively see \cite[Theorem 4.1]{S}):
\begin{lemma}\label{Matr-Bi-Free}
Let $(\mathcal A,\varphi)$ be a noncommutative probability space 
and $n\in\mathbb N$. Assume that $(a_1,b_1)$ and $(a_2,b_2)$ 
are bi-free with respect to $\varphi$. Then $(L(X_1),R(Y_1))$ and 
$(L(X_2),R(Y_2))$ are bi-free with amalgamation over $M_n(\mathbb C)$ 
whenever $X_j\in M_n(\mathbb C[a_j]),Y_j\in M_n(\mathbb C[b_j])$,
$j=1,2$. 
\end{lemma}

\section{Bi-free analytic subordination}\label{sec:bifreesubord}

In this section we establish our main subordination result for 
scalar-valued bi-free random variables. 
Fix a $C^*$-noncommutative probability space $(\mathcal A,\varphi)$ and two pairs
$(a_1,b_1),(a_2,b_2)\in\mathcal A^2$ that are bi-free with
respect to $\varphi$ and bi-partite. Thus, $a_j$ (respectively, $b_j$)
is the left (respectively, right) variable in the pair $(a_j,b_j)$,
and $a_jb_j=b_ja_j$, $j=1,2$. (Many of the computations below 
are valid under weaker hypotheses. In many circumstances, 
our computations with analytic transforms also
hold for operator-valued bi-free pairs of random variables.
We indicate below those cases in which this extension is valid.)

Define $X_j\in M_2(\mathcal A)$ by $X_j=
\begin{bmatrix}
a_j & 0\\
0 & b_j\end{bmatrix}$, $j=1,2$. In the context we will consider below, it is 
important that the left face is in the upper left, and the right face 
in the lower right, corner. We define the conditional 
expectation $M_2(\varphi)=\varphi\otimes{\rm Id}_{M_2(\mathbb C)}$
from $M_2(\mathcal A)$ onto $M_2(\mathbb C)$. With respect to
this expectation, we consider $G_{X_j}(v)=
E\left[(v-X_j)^{-1}\right]$ for $v\in M_2(\mathbb C)$ such that 
$\Im v>0$ or $\|v^{-1}\|<\|X_j\|^{-1}$. As usual, 
the $R$-transform is defined via the functional
equation $G_{X_j}(v^{-1}+R_{X_j}(v))=v$. As before, denote $K_{X_j}
(v)=v^{-1}+R_{X_j}(v)$. We restrict all of these 
functions to upper triangular matrices in $M_2(\mathbb C)$. 
Direct 
computation using the Schur complement yields
\begin{eqnarray*}
G_{X_j}\left(\begin{bmatrix}
z & \zeta\\
0 & w
\end{bmatrix}\right) & = &
M_2(\varphi)\begin{bmatrix}
(z-a_j)^{-1} & -(z-a_j)^{-1}\zeta(w-b_j)^{-1}\\
0 & (w-b_j)^{-1}
\end{bmatrix}\\
& = & \begin{bmatrix}
\varphi\left((z-a_j)^{-1}\right) & -\varphi
\left((z-a_j)^{-1}\zeta(w-b_j)^{-1}\right)\\
0 & \varphi\left((w-b_j)^{-1}\right)
\end{bmatrix}\\
& = & \begin{bmatrix}
G_{\mu_j}(z) & -\zeta G_{\eta_j}(z,w)\\
0 & G_{\nu_j}(w)
\end{bmatrix}.
\end{eqnarray*}
(First two equalities hold for operator-valued random variables
$(a_j,b_j)$, with the map $\varphi
\left((z-a_j)^{-1}\zeta(w-b_j)^{-1}\right)$ replaced by
$\zeta\mapsto M_{(a_j,b_j)}(z^{-1},z^{-1}\zeta w^{-1},w^{-1})
=G_{(a_j,b_j)}(z,\zeta,w)$ - see Section \ref{back}.) The 
compositional inverse of an analytic map that maps  
upper triangular matrices onto upper triangular matrices
preserves upper triangular matrices. That  is,
$K_{X_j}\left(\begin{bmatrix}
z & \zeta\\
0 & w
\end{bmatrix}\right)=
\begin{bmatrix}
K_{a_j}(z) & -k(z,\zeta,w)\\
0 & K_{b_j}(w)
\end{bmatrix}=\begin{bmatrix}
K_{\mu_j}(z) & -k(z,\zeta,w)\\
0 & K_{\nu_j}(w)
\end{bmatrix}$ for some function $k(z,\zeta,w)$.
Then, by the above formula, 
\begin{eqnarray*}
\lefteqn{\begin{bmatrix}
z & \zeta\\
0 & w
\end{bmatrix}= G_{X_j}\left(
\begin{bmatrix}
K_{a_j}(z) & -k(z,\zeta,w)\\
0 & K_{b_j}(w)
\end{bmatrix}\right)}\\
& = & \begin{bmatrix}
G_{a_j}(K_{a_j}(z)) &\varphi((K_{a_j}(z)-a_j)^{-1}k(z,\zeta,w)
(K_{b_j}(w)-b_j)^{-1}))\\
0 & G_{b_j}(K_{b_j}(w))
\end{bmatrix}\\
& = & \begin{bmatrix}
G_{a_j}(K_{a_j}(z)) &G_{(a_j,b_j)}(K_{a_j}(z),
k(z,\zeta,w),K_{b_j}(w)))\\
0 & G_{b_j}(K_{b_j}(w))
\end{bmatrix}\\
& = & \begin{bmatrix}
G_{\mu_j}(K_{\mu_j}(z)) & k(z,\zeta,w)G_{\eta_j}(K_{\mu_j}(z),
K_{\nu_j}(w))\\
0 & G_{\nu_j}(K_{\nu_j}(w))
\end{bmatrix}.
\end{eqnarray*}
(Again, the first three equalities hold for operator-valued random variables.)
The $(1,2)$ entry shows that
$k(z,\zeta,w)$ is linear in $\zeta$, and has as inverse
the linear map $\xi\mapsto G_{(a_j,b_j)}(K_{a_j}(z),
\xi,K_{b_j}(w))$.
Thus, 
\begin{eqnarray*}
k(z,\cdot,w) & = & G_{(a_j,b_j)}(K_{a_j}(z),
\cdot,K_{b_j}(w))^{\langle-1\rangle}\\
& = & K_{\mu_j}(z)
M_{(a_j,b_j)}(K_{\mu_j}(z)^{-1},\cdot,K_{\nu_j}(w)^{-1})^{\langle-1
\rangle}K_{\nu_j}(w)\\
& = & \Psi_{(a_j,b_j)}(K_{\mu_j}(z)^{-1},\cdot,K_{\nu_j}(w)^{-1}),
\end{eqnarray*}
and hence
$$
k(z,\zeta,w)=
\Psi_{(a_j,b_j)}(K_{\mu_j}(z)^{-1},\zeta,K_{\nu_j}(w)^{-1})
=\frac{\zeta}{G_{\eta_j}(K_{\mu_j}(z),
K_{\nu_j}(w))}.
$$ 
(The last quantity above only makes sense for scalar-valued variables.)
In particular,
\begin{eqnarray}
R_{X_j}\left(\begin{bmatrix}
z & \zeta\\
0 & w
\end{bmatrix}\right) & = &
\begin{bmatrix}
K_{\mu_j}(z)-z^{-1} & -k(z,\zeta,w)+z^{-1}\zeta w^{-1}\\
0 & K_{\nu_j}(w)-w^{-1}
\end{bmatrix}\nonumber\\
& = & \begin{bmatrix}
R_{\mu_j}(z) &z^{-1}\zeta w^{-1}-\Psi_{(a_j,b_j)}(K_{\mu_j}(z)^{-1},\zeta,K_{\nu_j}(w)^{-1})\\
0 & R_{\nu_j}(w)
\end{bmatrix}\nonumber\\
& = & \begin{bmatrix}
R_{\mu_j}(z) &\zeta\left(\frac{1}{zw}-\frac{1}{G_{\eta_j}(K_{\mu_j}(z),
K_{\nu_j}(w))}\right)\\
0 & R_{\nu_j}(w)
\end{bmatrix}\label{RR}
\end{eqnarray}
(with the first two equalities making sense for operator-valued variables).
Observe that in all the computations this far we have not
used the fact that $a_jb_j=b_ja_j$. If we
agree to consider only the band moments of type $\varphi(LR)$
($L$ and $R$ being monomials in the left, respectively right,
variable), all the above computations remain valid, with the 
possible difference that $\varphi\left((z-a_j)^{-1}(w-b_j)^{-1}\right)$
might not be the Cauchy transform of a probability measure $\eta_j$
on $\mathbb R^2$. We record our conclusion in the following:

\begin{lemma}\label{lem1}
Let $(\mathcal A,\varphi)$ be a $C^*$-noncommutative probability space and 
$(a,b)\in\mathcal A^2$ be a two-faced pair of noncommutative random variables. Define
the $M_2(\mathbb C)$-valued random variable $X=\begin{bmatrix}
a & 0 \\
0 & b
\end{bmatrix}\in M_2(\mathcal A)$. Then 
$$
R_{X}\left(\begin{bmatrix}
z & \zeta\\
0 & w
\end{bmatrix}\right)=\begin{bmatrix}
R_a(z) & \frac{\zeta}{zw}(R_{(a,b)}(z,w)-zR_a(z)-wR_b(w))\\
0 & R_b(w)
\end{bmatrix}.
$$
In particular, if $(a_1,b_1),(a_2,b_2)\in\mathcal A^2$ are bi-free
with respect to $\varphi$, then, with the notations from the
beginning of this section,
\begin{equation}\label{upper}
R_{X_1+X_2}\left(\begin{bmatrix}
z & \zeta\\
0 & w
\end{bmatrix}\right)=
R_{X_1}\left(\begin{bmatrix}
z & \zeta\\
0 & w
\end{bmatrix}\right)+R_{X_2}\left(\begin{bmatrix}
z & \zeta\\
0 & w
\end{bmatrix}\right)
\end{equation}
for all $z,w\in\mathbb C$ of sufficiently small absolute value and
all $\zeta\in\mathbb C$.

\end{lemma}

(Lemma \ref{lem1} holds for operator-valued random variables in a 
$B$-$B$-$C^*$-noncommutative probability space $(M, E, B)$. More precisely,
if $(a,b)\in M^2$ is a two-faced pair of $B$-valued random variables, then,
defining $X=\begin{bmatrix}
a & 0 \\
0 & b
\end{bmatrix}\in M_2(M)$, we have
$$
R_{X}\left(\begin{bmatrix}
z & \zeta\\
0 & w
\end{bmatrix}\right)=\begin{bmatrix}
R_a(z) & z^{-1}\zeta w^{-1}-
\Psi_{(a,b)}(K_{a}(z)^{-1},\zeta,K_{b}(w)^{-1})\\
0 & R_b(w)
\end{bmatrix},
$$
where $\Psi_{(a,b)}$ has been defined in Section \ref{back},
and Equation \eqref{upper} also holds in this context.)

We would like to emphasize again that $R_X$ above denotes the
$R$-transform of the $M_2(\mathbb C)$ (or $M_2(B)$)-valued 
random variable $X$, as introduced in \cite{V*}.
Thus, Equation \eqref{upper} implies that $X_1$ and $X_2$ ``mimic''
freeness in terms of the relations between their analytic
transforms when restricted to upper triangular matrices. More 
specifically:

\begin{remark}\label{rem2}
Let $(\mathcal A,\varphi)$ be a $C^*$-noncommutative
probability space. Assume that $(a_1,b_1)$ and $(a_2,b_2)$ are 
self-adjoint and bi-free with respect to $\varphi$. Define 
$X_j=\begin{bmatrix}
a_j & 0 \\
0 & b_j
\end{bmatrix}$, $j=1,2$ to be two self-adjoint random variables in the 
operator-valued noncommutative probability space 
$(M_2(\mathcal A),\varphi\otimes{\rm Id}_{M_2(\mathbb C)},
M_2(\mathbb C))$. Consider two random variables $Y_1,Y_2$ which are
free with respect to $M_2(\varphi):=\varphi\otimes{\rm Id}_{M_2(\mathbb C)}$, and 
such that the $*$-distribution of $X_j$ and $Y_j$ with respect to
$M_2(\varphi)$ coincide for $j=1,2$. 
Then the restrictions of the Cauchy and $R$-transforms of $X_1,X_2,
X_1+X_2$ and $Y_1,Y_2,Y_1+Y_2$, respectively, to the upper triangular 
$2\times2$ complex matrices with positive imaginary part coincide. The same statement holds for their restrictions
to the set of upper triangular $2\times2$ complex matrices having inverse of 
small norm. In particular, if $\omega_{Y_j}$ is the
$M_2(\mathbb C)$-valued subordination function satisfying
$G_{Y_j}\circ\omega_{Y_j}=G_{Y_1+Y_2}$, then 
$$
(G_{X_j}\circ\omega_{Y_j})\left(\begin{bmatrix}
z & \zeta\\
0 & w
\end{bmatrix}\right)^{-1}\!\!\!=G_{X_1+X_2}\left(\begin{bmatrix}
z & \zeta\\
0 & w
\end{bmatrix}\right)^{-1}\!\!\!=(\omega_{Y_1}+\omega_{Y_2})
\left(\begin{bmatrix}
z & \zeta\\
0 & w
\end{bmatrix}\right)-\begin{bmatrix}
z & \zeta\\
0 & w
\end{bmatrix},
$$
for all $z,w\in\mathbb C^+,\zeta\in\mathbb C,j=1,2$.
Thus,  
$\omega_{Y_j}=G_{Y_j}^{\langle-1\rangle}\circ G_{Y_1+Y_2}
=G_{X_j}^{\langle-1\rangle}\circ G_{X_1+X_2}$ when restricted
to matrices $\begin{bmatrix} z & \zeta\\0 & w\end{bmatrix}$ 
with $z,w$ having inverses of small norm.
(All relations above hold as well for operator-valued variables.)
\end{remark}

In light of Lemma \ref{lem1} and Equation \eqref{subord}, the proof 
of the above remark is obvious as soon as one accounts for the 
fact that the $M_2(\mathbb C)$-valued distribution of the
variable $X$ determines the $M_2(\mathbb C)$-valued Cauchy transform
of $X$. This remark allows us to recover Equation \eqref{four}.
Indeed, as it follows from \cite[Theorem 2.7]{BMS} that
the functions $\omega_{Y_j}$ map upper triangular
matrices to upper triangular matrices, we have that $\omega_{Y_j}
\left(\begin{bmatrix}
z & \zeta\\
0 & w
\end{bmatrix}\right)=\begin{bmatrix}
f_1 & f_2\\
0 & f_3
\end{bmatrix}.$ Since $X_j$ and $Y_j$ have the same distribution
and $X_j$ is diagonal and self-adjoint, so must be $Y_j$, and its
diagonal entries must have the same (joint) distribution as 
$(a_j,b_j)$. Thus, we obtain that
\begin{eqnarray*}
G_{Y_1+Y_2}\left(\begin{bmatrix}
z & \zeta\\
0 & w
\end{bmatrix}\right)&=&G_{X_j}\left(\begin{bmatrix}
f_1 & f_2\\
0 & f_3
\end{bmatrix}\right)\\
& = & \begin{bmatrix}
\varphi\left((f_1-a_j)^{-1}\right) & -\varphi\left((f_1-a_j)^{-1}f_2
(f_3-b_j)^{-1}\right)\\
0 & \varphi\left((f_3-b_j)^{-1}\right)
\end{bmatrix},
\end{eqnarray*}
which guarantees that $f_1=\omega_{a_j}(z),f_3=\omega_{b_j}(w)$.
Using again the previous remark, we obtain for $f_2=f_2(z,\zeta,w)$ and $z,w$ having small inverse
\begin{eqnarray*}
f_2 & = & \Psi_{(a_j,b_j)}\left(\omega_{a_1}(z)^{-1},\varphi\left(
(z-a_1-a_2)^{-1}\zeta(w-b_1-b_2)^{-1}\right),\omega_{b_1}(w)^{-1}
\right)\\
& = & \frac{\zeta
\varphi\left((z-a_1-a_2)^{-1}(w-b_1-b_2)^{-1}\right)}{\varphi\left(
(\omega_{a_j}(z)-a_j)^{-1}(\omega_{b_j}(w)-b_j)^{-1}\right)}.
\end{eqnarray*}
(Again, the first equality is true for operator-valued random variables.)
Recalling that $\omega_{Y_j}$ is defined and analytic on all matrices
$\begin{bmatrix}
z & \zeta\\
0 & w
\end{bmatrix}$ with positive imaginary part, we can write
$f_2(z,\zeta,w)G_{\eta_j}(\omega_{a_j}(z),\omega_{b_j}(w))=\zeta
G_{\eta_1\bboxplus\eta_2}(z,w)$. By analytic continuation, this relation 
holds for all $z,w\in\mathbb C^+$ and $\zeta\in\mathbb C$. In particular, 
we have $G_{\eta_1\bboxplus\eta_2}(z,w)=0$ whenever
$G_{\eta_j}(\omega_{a_j}(z),\omega_{b_j}(w))=0$, so that each connected component
in $\mathbb C^+\times\mathbb C^+$ of the zero set of $G_{\eta_j}(\omega_{a_j}(z),\omega_{b_j}(w))$
equals a connected component of the zero set of $G_{\eta_1\bboxplus\eta_2}$.
Replacing this in the upper triangular matrix-valued analogue (provided
above) of \eqref{subord} provides a slightly modified version of  Equation \eqref{four}:
\begin{equation}\label{subord12}
\frac{G_{\eta_1\bboxplus\eta_2}(z,w)}{G_{\eta_1}(\omega_{a_1}(z),\omega_{b_1}(w))}+\frac{G_{\eta_1\bboxplus\eta_2}(z,w)}{G_{\eta_2}(\omega_{a_2}(z),\omega_{b_2}(w))}
=\frac{G_{\eta_1\bboxplus\eta_2}(z,w)}{G_{\mu_1\boxplus\mu_2}(z)G_{\nu_1\boxplus\nu_2}(w)}+1,
\end{equation}
for all $z,w\in\mathbb C^+$, guaranteeing analytic extension through the zero sets
(in $\mathbb C^+\times\mathbb C^+$) of $G_{\eta_j}(\omega_{a_j}(z),\omega_{b_j}(w)),j=1,2$. (Observe that,
while in order to obtain a relation between Cauchy transforms of
measures in the plane, we need to assume that $a_j$ and $b_j$ commute,
the formal calculations above hold even in the absence of this 
hypothesis.) Thus, relation \eqref{subord12} holds on the connected set
$(\mathbb C^+\times\mathbb C^+)\cup(\mathbb C^-\times\mathbb C^-)
\cup\{(z,w)\in\mathbb C^2\colon|z|>\|a_1+a_2\|\text{ and }|w|>\|b_1+b_2\|\}$.
There is one obstacle to this relation extending to all of
$\mathbb C^+\times\mathbb C^-$, namely the possibility that 
$\frac{G_{\eta_j}(\omega_{a_j}(z_0),\omega_{b_j}(w_0))}{G_{\eta_1\bboxplus\eta_2}(z_0,w_0)}=0$ for some
$z_0\in\mathbb C^+,w_0\in\mathbb C^-,j\in\{1,2\}$. In this case, we have version \eqref{univ}
of equations \eqref{four} and \eqref{subord12} for computing $G_{\eta_1\bboxplus\eta_2}$ 
in terms of $G_{\eta_j},\omega_{a_j}$, and $\omega_{b_j},j=1,2$.

\begin{ex}
Relation \eqref{upper} does not extend to arbitrary $2\times2$
matrices $\begin{bmatrix}
z & \zeta\\
\zeta' & w
\end{bmatrix}$. In other words, $(a_1,b_1)$ and $(a_2,b_2)$ being bi-free
with respect to $\varphi$ does not necessarily imply that $X_1$ and $X_2$ are 
free with amalgamation over $M_2(\mathbb C)$, as it can be seen by computing 
some moments. Consider 
$$
M_2(\varphi)\left(\begin{bmatrix}
a_1 & 0 \\
0 & b_1
\end{bmatrix}\begin{bmatrix}
a_2 & 0 \\
0 & b_2
\end{bmatrix}\begin{bmatrix}
0 & 1 \\
0 & 0
\end{bmatrix}\begin{bmatrix}
a_2 & 0 \\
0 & b_2
\end{bmatrix}\begin{bmatrix}
a_1 & 0 \\
0 & b_1
\end{bmatrix}
\right)=\begin{bmatrix}
0 & \varphi(a_1a_2b_2b_1) \\
0 & 0
\end{bmatrix}.
$$
If $(a_1, b_1)$ and $(a_2,b_2)$ are bi-free with respect
to $\varphi$, then 
\begin{equation}\label{bi}
\varphi(a_1a_2b_2b_1) = \varphi(a_1b_1)\varphi(a_2)\varphi(b_2) + 
\varphi(a_2b_2)\varphi(a_1)\varphi(b_1) - \varphi(a_1)\varphi(a_2)\varphi(b_1)\varphi(b_2).
\end{equation}
On the other hand, if $\begin{bmatrix}
a_1 & 0 \\
0 & b_1
\end{bmatrix}$ and $\begin{bmatrix}
a_2 & 0 \\
0 & b_2
\end{bmatrix}$ are free with amalgamation over $M_2(\mathbb C)$, then
we would have
$$
M_2(\varphi)\left(\begin{bmatrix}
a_1 & 0 \\
0 & b_1
\end{bmatrix}\begin{bmatrix}
a_2 & 0 \\
0 & b_2
\end{bmatrix}\begin{bmatrix}
0 & 1 \\
0 & 0
\end{bmatrix}\begin{bmatrix}
a_2 & 0 \\
0 & b_2
\end{bmatrix}\begin{bmatrix}
a_1 & 0 \\
0 & b_1
\end{bmatrix}
\right)_{1,2}=\varphi(a_1b_1)\varphi(a_2b_2),
$$
which is generally different from the right hand-side of \eqref{bi}.
\end{ex}

\begin{remark}\label{rmk:biBoolean}
As mentioned in the introduction, the idea of constructing an $M_2(\mathbb C)$-valued 
random variable $X_j$ from the two-faced pair $(a_j, b_j)$ and considering the 
operator-valued transforms of $X_j$ was motivated by the study of bi-Boolean 
independence \cite{GS2}. We would like to elaborate a bit on this point. 
Given a probability measure $\mu$ on $\mathbb R$, define the analytic function 
\begin{equation}\label{h}
h_\mu\colon\mathbb{C}^+\to\mathbb C^+\cup\mathbb R,\quad h_\mu(z)=\frac{1}{G_\mu(z)}-z.
\end{equation} 
It was shown by Speicher and Woroudi in \cite{SW} that $h_\mu$ linearizes the Boolean additive convolution $\uplus$ 
in the sense that $h_{\mu_1\uplus\mu_2}(z) =h_{\mu_1}(z) +h_{\mu_2}(z),z\in\mathbb C^+$. 
Moreover, it was shown in \cite[Theorem 3.6]{SW} that every probability measure on $\mathbb R$ 
is $\uplus$-infinitely divisible. This result was later extended by Popa and Vinnikov to operator-valued
distributions in \cite{PV}. Given $X = X^*$ in a $W^*$-noncommutative probability space $(M,E,B)$, 
define $h_X$ by 
\begin{equation}\label{hop}
h_X(b)=G_X(b)^{-1}-b\quad b\in B,\Im b > 0\text{ or }\|b^{-1}\| < \|X\|^{-1}.
\end{equation} 
If $X_1$ and $X_2$ are Boolean independent with respect to $E$, then $h_{X_1 + X_2}(b)=h_{X_1}(b) +h_{X_2}(b)$. 
More importantly, it follows from \cite[Theorem 3.5]{PV} that the result on $\uplus$-infinite divisibility 
also holds in the $B$-valued setting (i.e. for every $n\in\mathbb{N}$, there exists $X_n = X_n^*$ such that $nh_{X_n}(b)=h_X(b)$).

On the other hand, given a probability measure $\eta$ on $\mathbb{R}^2$ 
with marginals $\mu$ and $\nu$, the function $\widetilde{E}_\eta$ was introduced in \cite[Section 4]{GS2} by
\begin{equation}\label{TildeE}
\widetilde{E}_\eta(z, w) = \frac{G_\eta(z, w)}{G_\mu(z)G_\nu(w)} - 1
, \quad (z, w) \in (\mathbb{C} \setminus \mathbb{R})^2,
\end{equation}
to study bi-Boolean independence.
The function $\widetilde{E}_\eta$ is the 
(reduced part of) the partial bi-Boolean self-energy which linearizes the bi-Boolean additive 
convolution $\uplus\uplus$ in the sense that $\widetilde{E}_\eta(z, w) = 
\widetilde{E}_{\eta_1}(z, w)+\widetilde{E}_{\eta_2}(z, w)$ if $\eta= \eta_1\uplus\uplus\eta_2$.
Due to the simple form of $\widetilde{E}_\eta$, it was natural to 
hypothesize that every probability measure on 
$\mathbb R^2$ is $\uplus\uplus$-infinitely divisible. 
When trying to prove the conjecture, we observed that if $(a, b)$ has joint distribution 
$\eta$ and if we consider the $M_2(\mathbb{C})$-valued random variable $X = \begin{bmatrix}
a & 0\\
0 & b
\end{bmatrix}$, then
\[h_X\left(\begin{bmatrix}
z & \zeta\\
0 & w
\end{bmatrix}\right) = \begin{bmatrix}
h_\mu(z) & \zeta\left(\frac{G_{\eta}(z, w)}{G_\mu(z)G_\nu(w)} - 1\right)\\
0 & h_\nu(w)
\end{bmatrix},\]
so that the three components of the partial bi-Boolean self-energy $E_\eta(z,w)=\widetilde{E}_\eta(z,w)-\frac1zh_\mu(z)-\frac1wh_\nu(w)$ 
appear as the three non-zero entries of 
$h_X$ evaluated at the upper triangular matrix $\begin{bmatrix}
z & \zeta\\
0 & w
\end{bmatrix}$. 
This, in connection to the above-mentioned result of Popa and Vinnikov, suggested that the conjecture might be true. 
However, that is not the case, as shown by the counterexample in \cite[Example 5.13]{GS2}. 
\end{remark}

The following consequence of Remark \ref{rem2} establishes some analytic properties
of the Cauchy transform of the bi-free additive convolution of
two probability measures in $\mathbb R^2$ that will help us
prove a converse of the characterization of bi-free extreme
values of Voiculescu.

\begin{prop}\label{Bound}
Let $(a_1, b_1)$ and $(a_2, b_2)$ be bi-free pairs of random variables. With the notations from Section {\rm \ref{back}}, we 
have 
$$
\quad\quad\quad\Im G_{a_j}(z)\Im G_{b_j}(w)\ge\Im z\Im w|G_{(a_j,b_j)}(z,w)|^2,
$$
and
$$
\Im\omega_{a_j}(z)\Im\omega_{b_j}(w)|G_{\eta_j}(\omega_{a_j}(z),
\omega_{b_j}(w))|^2\ge\Im z\Im w|G_{\eta_1\bboxplus\eta_2}(z,w)|^2,\quad
z,w\in\mathbb C^+.
$$
\end{prop}
\begin{proof}
The proof is based on a simple trick, which is a particular case of 
\cite[Proposition 3.1]{B}. Observe that $\Im \begin{bmatrix}
z & \zeta\\
0 & w
\end{bmatrix}>0$ if and only if $\Im z>0,\Im w>0$ and $4\Im z\Im w>|\zeta|^2$.
As $G_{X_j}$ maps elements from $M_2(\mathbb C)$ with positive imaginary part 
into elements from $M_2(\mathbb C)$ with negative imaginary part, it follows that
$4\Im G_{a_j}(z)\Im G_{b_j}(w)>|\zeta|^2|G_{(a_j,b_j)}(z,w)|^2$ whenever 
$\Im \begin{bmatrix}
z & \zeta\\
0 & w
\end{bmatrix}>0$. Letting $|\zeta|$ tend to $2\sqrt{\Im z\Im w}$ from below
yields
$$
\Im G_{a_j}(z)\Im G_{b_j}(w)\ge\Im z\Im w|G_{(a_j,b_j)}(z,w)|^2,\quad
z,w\in\mathbb C^+.
$$
The second relation follows from the fact that $\Im \omega_{X_j}\left(\begin{bmatrix}
z & \zeta\\
0 & w
\end{bmatrix}\right)>0$ whenever $\Im \begin{bmatrix}
z & \zeta\\
0 & w
\end{bmatrix}>0$, and the fact that
$$
\omega_{X_j}\left(\begin{bmatrix}
z & \zeta\\
0 & w
\end{bmatrix}\right)=\left(\begin{bmatrix}
\omega_{a_j}(z) & \frac{\zeta G_{\eta_1\bboxplus\eta_2}(z,w)}{G_{\eta_j}(\omega_{a_j}(z),\omega_{b_j}(w))}\\
0 & \omega_{b_j}(w)
\end{bmatrix}\right),
$$
as shown in the computations following Remark \ref{rem2}.
Indeed, it is known from \cite{V2000} that the maps $\omega$ 
introduced in \eqref{subord} send the set of elements with 
positive imaginary part into itself.
We obtain again that $\Im z>0,\Im w>0$ and $4\Im z\Im w>|\zeta|^2$ imply together
that $4\Im\omega_{a_j}(z)\Im\omega_{b_j}(w)>\left|
\frac{\zeta G_{\eta_1\bboxplus\eta_2}(z,w)}{G_{\eta_j}(\omega_{a_j}(z),\omega_{b_j}(w))}\right|^2$.
Letting $|\zeta|$ tend to $2\sqrt{\Im z\Im w}$ from below
allows us to conclude.
\end{proof}

We emphasize that this proposition allows us to make sense of Equation \eqref{four} as an equality of meromorphic functions on 
$\mathbb C^+\times\mathbb C^+$, and, given the properties of the one- and two-variables Cauchy transform, also on
$\mathbb C^-\times\mathbb C^-$. For the purposes of this article, we call a function $H(z,w)$ meromorphic if for any
$(z_0,w_0)$ in its domain we either have that $H$ is holomorphic around $(z_0,w_0)$, or both $z\mapsto H(z,w_0)$
and $w\mapsto H(z_0,w)$ are one-variable meromorphic functions on some neighborhood of $z_0$, and $w_0$, 
respectively. We have shown that the embodiment of Equation \eqref{four} in the shape of Equation \eqref{subord12}
makes sense as an equality of analytic functions; in particular, the zero set of $G_{\eta_1\bboxplus\eta_2}$ includes 
the zero sets of both $G_{\eta_j}(\omega_{a_j}(z),\omega_{b_j}(w)),j=1,2$. However, thanks to the previous proposition 
and Equation \eqref{subord12}, we can conclude more: if $G_{\eta_1\bboxplus\eta_2}(z_0,w_0)=0$, then
 $G_{\eta_1}(\omega_{a_1}(z_0),\omega_{b_1}(w_0))G_{\eta_2}(\omega_{a_2}(z_0),\omega_{b_2}(w_0))=0$. 
Indeed, let us assume that $G_{\eta_1}(\omega_{a_1}(z_0),\omega_{b_1}(w_0))\neq0$. By approaching $(z_0,w_0)$ 
from the complement of the zero sets of the three functions involved, we obtain 
$$
\lim_{(z,w)\to(z_0,w_0)}\frac{G_{\eta_1\bboxplus\eta_2}(z,w)}{G_{\eta_2}(\omega_{a_2}(z),\omega_{b_2}(w))}=1,
$$
so that $G_{\eta_2}(\omega_{a_2}(z_0),\omega_{b_2}(w_0))=0$. As shown in Section \ref{at}, 
$z\mapsto G_{\eta_1\bboxplus\eta_2}(z,w_0)$, $z\mapsto G_{\eta_j}(\omega_{a_j}(z),\omega_{b_j}(w_0))$
are not identically zero, so that $z_0$ is a zero of finite order for all these functions. The same statement
holds for $w$. If both $z\mapsto G_{\eta_j}(\omega_{a_j}(z),\omega_{b_j}(w_0))$, $j=1,2$, have a 
zero of precisely the same order at $z_0$, then so does $z\mapsto G_{\eta_1\bboxplus\eta_2}(z,w_0)$,
and relation \eqref{four} makes sense for those meromorphic functions when written as Laurent series 
around $z=z_0$. If one (say $z\mapsto
G_{\eta_2}(\omega_{a_2}(z),\omega_{b_2}(w_0))$ has a zero of higher order than the other, then
$z\mapsto G_{\eta_1\bboxplus\eta_2}(z,w_0)$ has a zero of the same order as $z\mapsto
G_{\eta_2}(\omega_{a_2}(z),\omega_{b_2}(w_0))$, so again \eqref{four} makes sense.

The main result of this section follows easily from the above 
considerations, Remark \ref{rem2} and Equation \eqref{seven}. 
\begin{theorem}\label{Main}
Let $(M, E, B)$ be a $C^*$-$B$-$B$-noncommutative probability space, for 
a $C^*$-algebra $B$.
Assume that $(a_1,b_1)$ and $(a_2,b_2)$ are  self-adjoint random 
variables in $M$ which are bi-free over $B$ with respect to 
$E$. Denote
$X_j=\begin{bmatrix}
a_j & 0 \\
0 & b_j
\end{bmatrix}$, $j=1,2$. Then 
\begin{eqnarray*}
\lefteqn{M_2(E)\left[\left(\begin{bmatrix}
z & \zeta\\
0 & w
\end{bmatrix}-X_1-X_2\right)^{-1}\right]=M_2(E)\left[\left(\begin{bmatrix}
\omega_{a_1}(z) & \Pi_1(z,\zeta,w)\\
0 & \omega_{b_1}(w)
\end{bmatrix}-X_1\right)^{-1}\right]}\quad\quad\quad\quad\quad\quad\quad\quad\quad\\
& = & 
\begin{bmatrix}
G_{a_1}(\omega_{a_1}(z)) & G_{(a_1,b_1)}(\omega_{a_1}(z),
\Pi_1(z,\zeta,w),\omega_{b_1}(w))\\
0 & G_{b_1}(\omega_{b_1}(w))
\end{bmatrix},
\end{eqnarray*}
where $\Pi_1(z,\cdot,w)$ is the composition of the following linear 
maps depending analytically on $z,w$:
\begin{eqnarray*}
\lefteqn{\Pi_1(z,\zeta,w)=
\left[G_{(a_1,b_1)}(\omega_{a_1}(z),\cdot,\omega_{b_1}(w))
+G_{(a_2,b_2)}(\omega_{a_2}(z),\cdot,\omega_{b_2}(w))\right.-}\\
& & \left.G_{(a_2,b_2)}(\omega_{a_2}(z),G_{a_1}(\omega_{a_1}(z))^{-1}
G_{(a_1,b_1)}(\omega_{a_1}(z),\cdot,\omega_{b_1}(w))
G_{b_1}(\omega_{b_1}(w))^{-1},\omega_{b_2}(w))\right]^{\langle-1\rangle}\\
& & \mbox{}\circ G_{(a_2,b_2)}(\omega_{a_2}(z),\zeta,\omega_{b_2}(w)).
\end{eqnarray*}
If $B=\mathbb C$, then $E$ is a state and the following simpler expression for $\Pi_1$ holds
$$
\Pi_1(z,\zeta,w)=\frac{\frac{\zeta}{G_{(a_1,b_1)}(\omega_{a_1}(z),
\omega_{b_1}(w))}}{\frac{1}{G_{(a_1,b_1)}(\omega_{a_1}(z),
\omega_{b_1}(w))}+\frac{1}{G_{(a_2,b_2)}(\omega_{a_2}(z),
\omega_{b_2}(w))}-\frac{1}{G_{a_1}(\omega_{a_1}(z))G_{b_1}(\omega_{b_1}(w))}}
$$
for all $z,w\in\mathbb C^+$, $\zeta\in\mathbb C$. Here $\omega_{a_1}$
and $\omega_{b_1}$ are the subordination functions introduced in 
Section 2, Equation \eqref{subord}, and $M_2(E)$ is the conditional 
expectation onto $M_2(B)$ given by $E\otimes{\rm Id}_{M_2(B)}$.
\end{theorem}
\begin{proof}
Given Equation \eqref{four}, we only need to argue that the expression for 
$\Pi_1$ makes sense. This follows directly from the considerations before the statement of our Theorem.
\end{proof}
Note that while the existence and analyticity of the operator-valued version of 
$\Pi_1$ is shown in Remark \ref{rem2} and the considerations following it, the
ingredients of its expression  as provided by the above theorem are guaranteed 
to be analytic only for $z^{-1},w^{-1}$ small in norm.

The following corollary is the converse of \cite[Theorem 2.1]{BiFreeExtreme}.

\begin{corollary}\label{nucular}
Assume that $(a_1,b_1)$ and $(a_2,b_2)$ are bi-free bi-partite 
self-adjoint two-faced pairs. We denote by $\eta_j$ the distribution
of $(a_j,b_j)$ and by $\mu_j$ and $\nu_j$ its first and second 
marginal, respectively. Assume that there exists a point
$(\xi,\zeta)\in\mathbb R^2$ such that $(\eta_1\bboxplus\eta_2)(\{
(\xi,\zeta)\})>0$. Then there exist $(\xi_j,\zeta_j)\in\mathbb R^2$,
$j=1,2,$ such that $(\xi_1+\xi_2,\zeta_1+\zeta_2)=(\xi,\zeta)$ and
$\eta_j(\{(\xi_j,\zeta_j)\})>0$. Moreover,
$$
1+\frac{(\mu_1\boxplus\mu_2)(\{\xi\})(\nu_1\boxplus\nu_2)(\{\zeta\})}{
(\eta_1\bboxplus\eta_2)(\{(\xi,\zeta)\})}=
\frac{\mu_1(\{\xi_1\})\nu_1(\{\zeta_1\})}{\eta_1(\{(\xi_1,\zeta_1)\})}+
\frac{\mu_2(\{\xi_2\})\nu_2(\{\zeta_2\})}{\eta_2(\{(\xi_2,\zeta_2)\})}.
$$
\end{corollary}
\begin{proof}
Observe first that 
$$(\mu_1\boxplus\mu_2)(\{\xi\})=
\int_\mathbb R{\bf 1}_{\{\xi\}\times\mathbb R}(s,t)\,{\rm d}
(\eta_1\bboxplus\eta_2)(s,t)\ge(\eta_1\bboxplus\eta_2)(\{(\xi,\zeta)\}
),$$
with a similar result for $\nu_1\boxplus\nu_2$. Thus,  
\cite[Theorem 7.4]{bv} indicates that
$\mu_j$ and $\nu_j$ all have atoms. More precise, there are $\xi_j,
\zeta_j\in\mathbb R,j=1,2,$ such that 
\begin{enumerate}
\item $1<\mu_1(\{\xi_1\})+\mu_2(\{\xi_2\})=
(\mu_1\boxplus\mu_2)(\{\xi\})+1$;
\item $1<\nu_1(\{\zeta_1\})+\nu_2(\{\zeta_2\})=
(\nu_1\boxplus\nu_2)(\{\zeta\})+1$;
\item $\xi_1+\xi_2=\xi;$
\item $\zeta_1+\zeta_2=\zeta$.
\end{enumerate}
According to the same article \cite{bv}, from this it follows that 
$\omega_{a_j}(iy+\xi)$, $\omega_{b_j}(iy+\zeta)$ tend nontangentially 
to $\xi_j$ and $\zeta_j$, respectively, as $y\downarrow0$, and
$$
\lim_{y\downarrow0}\frac{\Im\omega_{a_j}(\xi+iy)}{y}=\frac{\mu_j
(\{\xi_j\})}{(\mu_1\boxplus\mu_2)(\{\xi\})},\quad
\lim_{y\downarrow0}\frac{\Im\omega_{b_j}(\zeta+iy)}{y}=\frac{\nu_j
(\{\zeta_j\})}{(\nu_1\boxplus\nu_2)(\{\zeta\})}.
$$ 
From the 
dominated convergence theorem, we know that for any finite measure
$\eta$ in the plane and any sequences $\{z_n\}_n, \{w_n\}_n\subset
\mathbb C^+$ which converge nontangentially to $\xi$ and $\zeta$,
respectively, we have
$$
\lim_{n\to\infty}(z_n-\xi)(w_n-\zeta)G_{\eta}(z_n,w_n)=
\eta(\{(\xi,\zeta)\}).
$$
In particular, under the hypothesis of our corollary, $\lim_{y\to0}(iy)^2G_{\eta_1\bboxplus\eta_2}(\xi+iy,\zeta+iy)
=(\eta_1\bboxplus\eta_2)(\{(\xi,\zeta)\})>0$, and so, for $y>0$ small enough, $G_{\eta_1\bboxplus\eta_2}(\xi+iy,\zeta+iy)\neq0$.
Applying Proposition \ref{Bound} with $z=\xi+iy,w=\zeta+iy$ and letting $y\downarrow0$ yields
$$
\left(
\frac{(\mu_1\boxplus\mu_2)(\{\xi\})(\nu_1\boxplus\nu_2)(\{\zeta\})}{\mu_j(\{\xi_j\})\nu_j(\{\zeta_j\})}
\right)^\frac12
\eta_j(\{(\xi_j,\zeta_j)\})\ge(\eta_1\bboxplus\eta_2)(\{
(\xi,\zeta)\})>0.
$$ 
We divide by $y^2$ in Equation \eqref{four} to conclude that
\begin{eqnarray*}
\lefteqn{\frac{1}{(\eta_1\bboxplus\eta_2)(\{
(\xi,\zeta)\})}+
\frac{1}{(\mu_1\boxplus\mu_2)(\{
\xi\})(\nu_1\boxplus\nu_2)(\{\zeta\})} }\\
& = & 
\lim_{y\downarrow0}\frac{1}{y^2G_{\eta_1\bboxplus
\eta_2}(\xi+iy,\zeta+iy)}+\frac{1}{yG_{\mu_1\boxplus\mu_2}(\xi+iy)
yG_{\nu_1\boxplus\nu_2}(\zeta+iy)}\\
& = & \sum_{j=1}^2\left[\lim_{y\downarrow0}
\frac{\Im\omega_{a_j}(\xi+iy)\Im\omega_{b_j}(\zeta+iy)}{y^2}\right.\\
& & \left.\mbox{}\times\lim_{y\downarrow0}
\frac{1}{\Im\omega_{a_j}(\xi+iy)\Im\omega_{b_j}(\zeta+iy)G_{\eta_j}
(\omega_{a_j}(\xi+iy),\omega_{b_j}(\zeta+iy))}\right]\\
& = & 
\frac{1}{(\mu_1\boxplus\mu_2)(\{\xi\})(\nu_1\boxplus\nu_2)(\{\zeta)\})}
\left(
\frac{\mu_1(\{\xi_1\})\nu_1(\{\zeta_1\})}{\eta_1(\{(\xi_1,\zeta_1)\})}+
\frac{\mu_2(\{\xi_2\})\nu_2(\{\zeta_2\})}{\eta_2(\{(\xi_2,\zeta_2)\})}
\right).
\end{eqnarray*}
\end{proof}

\begin{remark} As an immediate consequence, \cite[Corollary 7.5]{bv} holds for probability measures on 
$\mathbb R^2$ and bi-free convolution: that is, for any compactly supported Borel probability measure
$\eta$ on $\mathbb R^2$, $\eta\bboxplus\eta$  has at most one atom, and if $\eta$ is $\bboxplus$-infinitely 
divisible (see \cite{GHM}), then it has at most one atom.
\end{remark}

Significantly stronger results are known for free additive convolution:
no singular continuous part is present in the Lebesgue decomposition 
of the free convolution of two probability measures whose supports contain
more than one point, and its absolutely continuous part has a density
which is continuous wherever finite (see \cite{B-Reg}). For compactly
supported measures, it has been shown in \cite{B-RR} that the Cauchy
transform of the free convolution of two probability measures $\mu$ and $\nu$
is unbounded if and only if $\mu(\{\alpha\})+\nu(\{\beta\})\ge1$ for some
$\alpha,\beta\in\mathbb R$. Thus, only atoms of $\mu$ and $\nu$ can make the density
of $\mu\boxplus\nu$ unbounded. Of course, a similar result cannot 
be expected to hold for bi-free convolutions (indeed, there are 
examples of central limits which are not absolutely continuous with 
respect to the Lebesgue measure on $\mathbb R^2$ - see \cite{BiFree1}). 
However, Corollary \ref{nucular} can be significantly (and easily) improved
in this direction, using the  tools provided by Proposition
\ref{Bound}.

\begin{corollary}
Assume that $(a_1,b_1)$ and $(a_2,b_2)$ are bi-free bi-partite 
self-adjoint two-faced pairs, neither of $a_1,b_1,a_2,b_2$ being a multiple of the identity. 
We denote by $\eta_j$ the distribution
of $(a_j,b_j)$ and by $\mu_j$ and $\nu_j$ its first and second 
marginal, respectively, $j=1,2$. Assume that there exist sequences $\{z_n\}_{n\in\mathbb N},
\{w_n\}_{n\in\mathbb N}\subset\mathbb C^+$ such that 
\begin{equation}
\lim_{n\to\infty}\sqrt{\Im z_n\Im w_n}|G_{\eta_1\bboxplus\eta_2}(z_n,w_n)|=+\infty.
\end{equation}
Then at least one of the two marginals of $\eta_1\bboxplus\eta_2$
has an unbounded Cauchy transform. In particular, there exists
a pair $\alpha,\beta\in\mathbb R$ such that either
$\mu_1(\{\alpha\})+\mu_2(\{\beta\})\ge1$, or
$\nu_1(\{\alpha\})+\nu_2(\{\beta\})\ge1$.
\end{corollary}
Note that this result provides a large class of probability
measures in $\mathbb R^2$ which cannot arise as non-trivial 
bi-free convolutions.

\begin{proof}
The proof is  straightforward: recall from Proposition \ref{Bound}
that 
\begin{eqnarray*}
\sqrt{\Im G_{\mu_1\boxplus\mu_2}(z_n)\Im G_{\nu_1\boxplus\nu_2}(w_n)} 
& \ge & \sqrt{\Im z_n\Im w_n}|G_{\eta_1\bboxplus\eta_2}(z_n,w_n)|.
\end{eqnarray*}
Thus, at least one of $\{G_{\mu_1\boxplus\mu_2}(z_n)\}_{n\in\mathbb N},\{G_{\nu_1\boxplus\nu_2}(w_n)\}_{n\in\mathbb N}$
is unbounded. An application of \cite[Theorem 7]{B-RR} concludes the proof.
\end{proof}

\begin{corollary}
With the notations, and under the hypotheses, of the previous corollary, 
assume that there exist $u,x\in\mathbb R$ and $\alpha>1/2$ such that 
$$
\lim_{(y_1,y_2)\to(0,0)}(y_1y_2)^\alpha|G_{\eta_1\bboxplus\eta_2}(x+iy_1,u+iy_2)|=+\infty.
$$
Then there exist pairs $\alpha_1,\alpha_2,\beta_1,\beta_2\in\mathbb R$ such that $\alpha_1+\beta_1=x,\alpha_2+\beta_2=u$, and  
$\mu_1(\{\alpha_1\})+\mu_2(\{\beta_1\})\ge1$, $\nu_1(\{\alpha_2\})+\nu_2(\{\beta_2\})\ge1$.
\end{corollary}

\begin{proof}
As seen above, we have
\begin{eqnarray*}
\lefteqn{
(y_1y_2)^{\alpha-\frac12}\left(\Im G_{\mu_1\boxplus\mu_2}(x+iy_1)\Im G_{\nu_1\boxplus\nu_2}(u+iy_2)\right)^{\frac12}}
\quad\quad\quad\quad\quad\quad\quad\quad\quad\quad\quad\quad\quad\quad\\ 
& \ge & (y_1y_2)^\alpha|G_{\eta_1\bboxplus\eta_2}(x+iy_1,u+iy_2)|.
\end{eqnarray*}
Taking limit as $(y_1,y_2)\to(0,0)$ in this inequality yields
$$
\lim_{(y_1,y_2)\to(0,0)}(y_1y_2)^{2\alpha-1}\Im G_{\mu_1\boxplus\mu_2}(x+iy_1)\Im G_{\nu_1\boxplus\nu_2}(u+iy_2)=+\infty
$$
Since the two coordinates $y_1,y_2$ in the above limit are independent of each other, we conclude that
there exists a $t\in(0,2\alpha-1]$ such that the sets 
$\{y^t\Im G_{\mu_1\boxplus\mu_2}(x+iy)\colon y>0\}$ and
$\{y^t\Im G_{\nu_1\boxplus\nu_2}(x+iy)\colon y>0\}$ are both unbounded.
An application of the same \cite[Theorem 7]{B-RR} allows us to conclude.
\end{proof}

We conclude this section with a simple remark generalizing
the linearization result from \cite{BMS} to bi-free bi-partite
self-adjoint random variables.

\begin{prop}
Assume that $(a_1,b_1),(a_2,b_2)\in\mathcal A^2$
self-adjoint random variables in the C${}^*$-noncommutative probability 
space $(\mathcal A,\varphi)$ which are bi-free with respect to 
$\varphi$ and satisfy $a_jb_j=b_ja_j$, $j=1,2$. 
Consider self-adjoint polynomials $p,q$ in two noncommuting
indeterminates. Then there exist $m\in\mathbb N$, $\alpha_j,\beta_j,
\gamma_j\in M_{m+1}(\mathbb C)$ which are self-adjoint such that
\begin{eqnarray*}
\lefteqn{\varphi\left((z-p(a_1,a_2))^{-1}(w-q(b_1,b_2))^{-1}\right)=}\\
& & [G_{(a_1\otimes\alpha_1+a_2\otimes\alpha_2,b_1\otimes\beta_1+
b_2\otimes\beta_2)}(ze_{1,1}+\gamma_1,-e_{1,1},we_{1,1}+\gamma_2)]_{1,m+2}.
\end{eqnarray*}
Moreover, $G_{(a_1\otimes\alpha_1+a_2\otimes\alpha_2,b_1\otimes\beta_1+
b_2\otimes\beta_2)}$ can be computed via the analytic subordination 
functions provided by Theorem {\rm \ref{Main}.}
\end{prop}
\begin{proof}
As shown in \cite{A} (see also \cite[Section 3]{BMS}), 
for any self-adjoint polynomial in two noncommuting self-adjoint indeterminates
$p\in\mathbb C\langle X_1,X_2\rangle$, one can find 
$m_p\in\mathbb N$ and a self-adjoint matrix $L_p=\begin{bmatrix}
0 & -u_p^*\\
-u_p & -Q_p
\end{bmatrix}\in M_{m_p+1}(\mathbb C\langle X_1,X_2\rangle)$
such that 
\begin{enumerate}
\item each entry of $L_p$ is of degree less than or equal to one, 
\item $u_p\in M_{m_p\times 1}(\mathbb C\langle X_1,X_2\rangle)$,
\item $Q_p\in M_{m_p}(\mathbb C\langle X_1,X_2\rangle)$ is invertible with
$Q_p^{-1}\in M_{m_p}(\mathbb C\langle X_1,X_2\rangle)$, and
\item $\left[(ze_{1,1}-L_p)^{-1}\right]_{1,1}=\left(\begin{bmatrix}
z & u_p^*\\
u_p & Q_p
\end{bmatrix}^{-1}\right)_{1,1}=(z-p)^{-1}$, that is, $p=u_p^*Q_p^{-1}u_p$.
\end{enumerate}
Clearly, such an $L_p$ needs not be unique. 

Choose now such a matrix $L_p\in M_{m_p+1}(\mathbb C\langle X_1,X_2\rangle)$,
and another matrix $L_q\in M_{m_q+1}(\mathbb C\langle X_1,X_2\rangle)$ satisfying the
same properties. We would like to evaluate $L_p$ in $X_1=a_1$, $X_2=a_2$ and
$L_q$ in $X_1=b_1,$ $X_2=b_2$ and apply Lemma \ref{Matr-Bi-Free} and 
Remark \ref{rem2}. In order to be able to do that, we need that $m_p=m_q$. Unfortunately there is no 
apriori reason for that to happen. Assume without loss of generality that $m_p<m_q$.
We show next that we can modify $L_p$ such that it still satisfies items (1)--(4) above,
but with $m_p$ replaced by $m_p+r$ for any $r\in\mathbb N$.  
Indeed, if $L_p$ satisfies (1)--(4) above, then
$$
\left(ze_{1,1}-\begin{bmatrix}
L_p & 0_{m_p\times r}\\
0_{r\times m_p} & 1_{r\times r}
\end{bmatrix}\right)^{-1}=\begin{bmatrix}
(ze_{1,1}-L_p)^{-1} & 0_{m_p\times r}\\
0_{r\times m_p} & -1_{r\times r}
\end{bmatrix}.
$$
Thus, if our first choice of $L_p$ and $L_q$ have different sizes, 
we complete the smaller one with an identity matrix of the 
desired size in the lower right corner (and zero elsewhere) 
in order to make them of equal size $m+1$. Then
\begin{eqnarray*}
\lefteqn{\begin{bmatrix}
z & u_p^* & -1 & 0_{1\times m}\\
u_p & Q_p & 0_{m\times1}& 0_{m\times m}\\
0 & 0_{1\times m} & w & u_q^*\\
0_{m\times1} & 0_{m\times m} & u_q & Q_q
\end{bmatrix}^{-1}}\quad\quad\quad\quad\quad\quad\\
& = & \begin{bmatrix}
(z-p)^{-1} & \star & (z-p)^{-1}(w-q)^{-1} & \star\\
\star &\star & \star & \star\\
0 & 0_{1\times m} & (w-q)^{-1} & \star\\
0_{m\times1} & 0_{m\times m} &\star & \star
\end{bmatrix}
\end{eqnarray*}
We evaluate $u_p,Q_p$ in $a_1$ and $a_2$, $u_q,Q_q$ in $b_1$ and $b_2$,
and apply $\varphi$. 
The proposition follows now easily from Lemma \ref{Matr-Bi-Free}
and Remark \ref{rem2}.
\end{proof}

\section{Bi-free convolution semigroups}\label{sec:semi}

One remarkable feature of free additive convolution 
is the existence of partially defined free convolution semigroups:
for any Borel probability measure $\mu$ on $\mathbb R$,
there exists a family $(\mu_t)_{t\ge1}$ of Borel
probability measures on $\mathbb R$ such that $\mu_1
=\mu$ and $\mu_{s+t}=\mu_s\boxplus\mu_t$ for all 
$s,t\ge1$. This phenomenon was first noted in 
\cite{BV} for $t$ large enough, and proved in \cite{NS}
for all $t\ge1$. In \cite{BB}, an analytic subordination 
formula for $\mu_t$ to $\mu_1$ is provided: for any $t\ge1$
and $z\in\mathbb C^+$, there exists 
$\omega_\mu(t,z)\in\mathbb C^+$ such that
$$
\omega_\mu(t,z)=\frac{z}{t}+\left(1-\frac1t\right)\frac{1}{G_{\mu}(\omega_\mu(t,z))}.
$$
Moreover, $G_{\mu}(\omega_\mu(t,z))=G_{\mu_t}(z)$, 
$z\in\mathbb C$, and the correspondence 
$z\mapsto\omega_\mu(t,z)$ is analytic on $\mathbb C^+$.
It is easy to see that the above equation uniquely
determines $\omega_\mu(t,z)$, and hence $\mu_t$.
The paper \cite{NS} provides also an operatorial 
construction of $\mu_t$: if $a=a^*$ in some 
${}^*$-noncommutative probability space $(\mathcal A,\varphi)$ has 
distribution $\mu$ with respect to $\varphi$ and $p=p^*=p^2$
is a projection which is free from $a$ and satisfies 
$\varphi(p)=1/t$, then the distribution of $pap$ in
the reduced algebra $(p\mathcal Ap,\frac{1}{\varphi(p)}
\varphi(p\cdot p))$ is $\mu_t$. Using this construction,
one of us has generalized, together with Huang and Mingo,
the result of Nica and Speicher to bi-free additive convolution.

More precisely, it has been shown in \cite[Theorem 5.3]{GHM} 
that for any compactly supported Borel probability measure 
$\eta$ on $\mathbb R^2$, there exists a partially defined 
bi-free convolution semigroup $(\eta_t)_{t\ge1}$ satisfying the
conditions $\eta_1=\eta$ and $\eta_{s+t}=\eta_s\bboxplus
\eta_t$ for all $s,t\ge1$. As expected, we have
$R_{\eta_t}(z,w)=tR_\eta(z,w)$, $t\ge1$. 
The partial semigroup $(\eta_t)_{t\ge1}$ extends to a full
weakly continuous semigroup $[0,+\infty)\ni t\mapsto\eta_t$
with $\eta_0=\delta_{(0,0)}$ if and only if 
$\eta$ is bi-freely infinitely divisible 
(see \cite[Theorem 4.2]{GHM}). 

Consider a $C^*$-noncommutative probability space $(\mathcal A,\varphi)$ and let $(a_1,b_1)\in\mathcal A^2$
be a bi-partite two-faced pair of self-adjoint random variables 
whose joint distribution with respect to $\varphi$ is $\eta$. For
any $t\ge1$, let $(a_t,b_t)\in\mathcal A^2$ be a two-faced pair 
of noncommutative random variables such that $a_tb_t=b_ta_t,
a_t=a_t^*$, and $b_t=b_t^*$, and the distribution of $(a_t,b_t)$
with respect to $\varphi$ equals $\eta_t$. Denote by $\mu_t$ the 
distribution of $a_t$ and by $\nu_t$ the distribution of $b_t$ with respect to $\varphi$.
Define $X_t=\begin{bmatrix}
a_t & 0 \\
0 & b_t
\end{bmatrix}\in M_2(\mathcal A)$. As seen in Lemma \ref{lem1},
we have
\begin{eqnarray*}
R_{X_t}\left(\begin{bmatrix}
z & \zeta \\
0 & w
\end{bmatrix}\right) & = & \begin{bmatrix}
R_{a_t}(z) & z^{-1}\zeta w^{-1}(R_{(a_t,b_t)}(z,w)-zR_{a_t}(z)-wR_{b_t}(w)) \\
0 & R_{b_t}(w)
\end{bmatrix}\\
& = & t\begin{bmatrix}
R_{a_1}(z) & z^{-1}\zeta w^{-1}(R_{(a_1,b_1)}(z,w)-zR_{a_1}(z)-wR_{b_1}(w)) \\
0 & R_{b_1}(w)
\end{bmatrix}\\
& = & tR_{X_1}\left(\begin{bmatrix}
z & \zeta \\
0 & w
\end{bmatrix}\right).
\end{eqnarray*}
As shown in \cite[Theorem 7.9]{ABFN}, for a given 
$X_1=X_1^*\in M_2(\mathcal A)$ and $t\ge1$, there exists 
an $\tilde{X}_t=\tilde{X}_t^*$ such that $R_{\tilde{X}_t}=
tR_{X_1}$. By restricting $R_{\tilde{X}_t}$ to the 
set of upper triangular matrices and applying Lemma \ref{lem1}, 
we see that $R_{\tilde{X}_t}\left(\begin{bmatrix}
z & \zeta \\
0 & w
\end{bmatrix}\right)=R_{X_t}\left(\begin{bmatrix}
z & \zeta \\
0 & w
\end{bmatrix}\right)$ for all $z,w,\zeta\in\mathbb C$ of
sufficiently small absolute value. In particular, it follows
that the subordination formula of \cite[Theorem 8.4]{ABFN}
holds  for $X_t$: there exists a function $\omega_{X_t}$ defined
on the set of elements $\begin{bmatrix}
z & \zeta \\
0 & w
\end{bmatrix}$ with strictly positive imaginary part which
satisfies the functional equation 
$$
\omega_{X_t}\left(\begin{bmatrix}
z & \zeta \\
0 & w
\end{bmatrix}\right)=\frac1t\begin{bmatrix}
z & \zeta \\
0 & w
\end{bmatrix}+\left(1-\frac1t\right)
G_{X_1}\left(\omega_{X_t}\left(\begin{bmatrix}
z & \zeta \\
0 & w
\end{bmatrix}\right)\right)^{-1},
$$
and $G_{X_1}\circ\omega_{X_t}=G_{X_t}$.
The point $\omega_{X_t}\left(\begin{bmatrix}
z & \zeta \\
0 & w
\end{bmatrix}\right)$ is the unique attracting
fixed point of the map
$v\mapsto\frac1t\begin{bmatrix}
z & \zeta \\
0 & w
\end{bmatrix}+\left(1-\frac1t\right)
G_{X_1}(v)^{-1}.$ Since this map sends
upper triangular matrices to upper triangular 
matrices, it follows that $\omega_{X_t}\left(\begin{bmatrix}
z & \zeta \\
0 & w
\end{bmatrix}\right)=\begin{bmatrix}
f_1(z,\zeta,w) & f_2(z,\zeta,w) \\
0 & f_3(z,\zeta,w)
\end{bmatrix}$ itself is upper triangular. Its
entries are easily determined by using the above-displayed
equation:
\begin{eqnarray*}
\lefteqn{\begin{bmatrix}
f_1(z,\zeta,w) & f_2(z,\zeta,w) \\
0 & f_3(z,\zeta,w)
\end{bmatrix}=\frac1t\begin{bmatrix}
z & \zeta \\
0 & w
\end{bmatrix}}\\
&  & \mbox{}+\left(1-\frac1t\right)\begin{bmatrix}
G_{a_1}(f_1(z,\zeta,w))^{-1} & \frac{\varphi\left((f_1(z,\zeta,w)-a_1)^{-1}f_2(z,\zeta,w)(f_3(z,\zeta,w)-b_1)^{-1}\right)}{G_{a_1}(f_1(z,\zeta,w))G_{b_1}(f_3(z,\zeta,w))} \\
0 & G_{b_1}(f_3(z,\zeta,w))^{-1}
\end{bmatrix}\\
& = & \begin{bmatrix}
\frac{z}{t}+\left(1-\frac1t\right)\frac{1}{G_{a_1}(f_1(z,\zeta,w))} & \frac1t\zeta+\left(1-\frac1t\right)\frac{f_2(z,\zeta,w)G_{(a_1,b_1)}(f_1(z,\zeta,w),f_3(z,\zeta,w))}{G_{a_1}(f_1(z,\zeta,w))G_{b_1}(f_3(z,\zeta,w))} \\
0 & \frac{w}{t}+\left(1-\frac1t\right)\frac{1}{G_{b_1}(f_3(z,\zeta,w))}
\end{bmatrix}.
\end{eqnarray*}
The equalities corresponding to entries 
$(1,1)$ and $(2,2)$ provide as indicated at
the beginning of this section, via \cite[Theorem 2.5]{BB},
that $f_1(z,\zeta,w)=\omega_{\mu_1}(t,z)$, $f_3(z,\zeta,w)
=\omega_{\nu_1}(t,w)$. The $(1,2)$ corner provides the relation
\begin{eqnarray*}
\lefteqn{f_2(z,\zeta,w)=\zeta\frac{G_{a_1}(\omega_{\mu_1}(t,z))G_{b_1}(\omega_{\nu_1}(t,w))}{
tG_{a_1}(\omega_{\mu_1}(t,z))G_{b_1}(\omega_{\nu_1}(t,w))+(1-t)
G_{(a_1,b_1)}(\omega_{\mu_1}(t,z),\omega_{\nu_1}(t,w))}}\\
& = & \zeta\frac{G_{a_t}(z)G_{b_t}(w)}{
tG_{a_t}(z)G_{b_t}(w)+(1-t)
G_{(a_1,b_1)}(\omega_{\mu_1}(t,z),\omega_{\nu_1}(t,w))}.
\quad\quad\quad\quad\quad\quad\quad\quad\quad\quad
\end{eqnarray*}
Thus, using $G_{X_1}\circ\omega_{X_t}=G_{X_t}$ we obtain a formula for the Cauchy transform 
of a measure in a partial bi-free additive convolution 
semigroup:
\begin{equation}\label{undici}
G_{(a_t,b_t)}(z,w)=\frac{1}{\frac{t}{G_{(a_1,b_1)}(\omega_\mu(t,z),\omega_\nu(t,w))}+\frac{1-t}{
G_{a_1}(\omega_\mu(t,z))G_{b_1}(\omega_\nu(t,w))}},\quad z,w\in\mathbb C^+.
\end{equation}
An analogue of Proposition \ref{Bound} now easily follows:

\begin{prop}\label{prop41}
Let $(\mathcal A,\varphi)$ be a $C^*$-noncommutative probability space. Assume
that for any $t\ge1$, there is a commuting self-adjoint two-faced pair 
$(a_t,b_t)\in\mathcal A^2$ of noncommutative random variables such that the distribution of $(a_t,b_t)$
with respect to $\varphi$ equals $\eta_t$, and $\eta_{s+t}=
\eta_s\bboxplus\eta_t$, $s,t\ge1$. Denote $\mu_t$ the 
distribution of $a_t$ and $\nu_t$ the distribution of $b_t$.
Then
$$
\Im \omega_\mu(t,z)\Im\omega_\nu(t,w)|G_{(a_1,b_1)}(\omega_\mu(t,z),\omega_\nu(t,w))|^2\ge
\Im z\Im w|G_{(a_t,b_t)}(z,w)|^2, 
$$
for all $z,w\in\mathbb C^+.$
\end{prop}
\begin{proof}
The inequality follows from the fact that $\Im\omega_{X_t}\left(\begin{bmatrix}
z & \zeta \\
0 & w
\end{bmatrix}\right)>0$ whenever $\Im\begin{bmatrix}
z & \zeta \\
0 & w
\end{bmatrix}>0$ in $M_2(\mathbb C)$ and from the
relation 
$$
G_{(a_t,b_t)}(z,w)=f_2(z,\zeta,w)G_{(a_1,b_1)}(\omega_{\mu_1}(t,z),\omega_{\nu_1}(t,w)).
$$
The proof is identical to the proof of Proposition \ref{Bound} and is left as an exercise to the reader.
\end{proof}

\begin{corollary}
Consider a compactly supported Borel probability measure $\eta$ on $\mathbb R^2$ and let 
$t>1$ be given. Let $(\eta_t)_{t\ge1}$ be its partial bi-free convolution semigroup.
Assume that there is a point $(\xi,\zeta)\in\mathbb R^2$ so that $\eta_t(\{(\xi,\zeta)\})>0$. Then
$\eta(\{(\xi/t,\zeta/t)\})>0$ and 
$$
\eta_t(\{(\xi,\zeta)\})=\frac{(t\mu(\{\xi/t\})+1-t)(t\nu(\{\zeta/t\})+1-t)
\eta(\{(\xi/t,\zeta/t)\})}{t\mu(\{\xi/t\})\nu(\{\zeta/t\})+(1-t)\eta(\{(\xi/t,\zeta/t)\})}
$$
\end{corollary}
\begin{proof}
The presence of an atom of $\eta_t$ at $(\xi,\zeta)$ implies the presence of atoms
for the marginals $\mu_t$ (at $\xi$) and $\nu_t$ (at $\zeta$), respectively.
Thus, as shown in \cite[Theorem 3.1]{BB}, we have
\begin{enumerate}
\item $\lim_{y\downarrow0}\omega_{\mu}(t,\xi+iy)=\xi/t$, $\lim_{y\downarrow0}\omega_{\nu}
(t,\zeta+iy)=\zeta/t$;
\item $\mu_t(\{\xi\})=t\mu(\{\xi/t\})+1-t,\nu_t(\{\zeta\})=t\nu(\{\zeta/t\})+1-t$;
\item $\lim_{y\downarrow0}\frac{\Im\omega_{\mu}(t,\xi+iy)}{y}=\frac1t+\left(1-\frac1t\right)
\frac{1}{\mu_t(\{\xi\})}=\frac{\mu(\{\xi/t\})}{t\mu(\{\xi/t\})+1-t}$ and
\newline\noindent 
$\lim_{y\downarrow0}\frac{\Im\omega_{\nu}(t,\zeta+iy)}{y}=\frac1t+\left(1-\frac1t\right)
\frac{1}{\nu_t(\{\zeta\})}=\frac{\nu(\{\zeta/t\})}{t\nu(\{\zeta/t\})+1-t}$.
\end{enumerate}
In particular, $\mu(\{\xi/t\})>1-1/t$ and $\nu(\{\zeta/t\})>1-1/t.$
Applying Proposition \ref{prop41} to $z=\xi+iy,w=\zeta+iy$ and taking limit as $y\to0$
we obtain
$$
\left(\frac{(t\mu(\{\xi/t\})+1-t)(t\nu(\{\zeta/t\})+1-t)}{\mu(\{\xi/t\})\nu(\{\zeta/t\})}\right)^\frac{1}{2}
\eta(\{(\xi/t,\zeta/t)\})\ge\eta_t(\{(\xi,\zeta)\})>0.
$$
Thus, $\eta(\{(\xi/t,\zeta/t)\})>0$. Multiplying by $y^2$ in \eqref{undici} evaluated in 
$z=\xi+iy,w=\zeta+iy$ and taking limits as $y$ decreases to zero yields
\begin{eqnarray*}
\lefteqn{\eta_t(\{(\xi,\zeta)\})}\\
 & = & \lim_{y\downarrow0}y^2G_{(a_t,b_t)}(\xi+iy,\zeta+iy)\\
& = & \lim_{y\downarrow0}\frac{1}{\frac{t\Im\omega_\mu(t,\xi+iy)\Im\omega_\nu(t,\zeta+iy)}{y^2\Im\omega_\mu(t,\xi+iy)\Im\omega_\nu(t,\zeta+iy)G_{(a_1,b_1)}(\omega_\mu(t,\xi+iy),\omega_\nu(t,\zeta+iy))}+\frac{1-t}{
yG_{a_t}(\xi+iy)yG_{b_t}(\zeta+iy)}}\\
& = & \frac{1}{\frac{t\mu(\{\xi/t\})\nu(\{\zeta/t\})}{(t\mu(\{\xi/t\})+1-t)(t\nu(\{\zeta/t\})+1-t)
\eta(\{(\xi/t,\zeta/t)\})}+\frac{1-t}{(t\mu(\{\xi/t\})+1-t)(t\nu(\{\zeta/t\})+1-t)}}\\
& = & \frac{(t\mu(\{\xi/t\})+1-t)(t\nu(\{\zeta/t\})+1-t)
\eta(\{(\xi/t,\zeta/t)\})}{t\mu(\{\xi/t\})\nu(\{\zeta/t\})+(1-t)\eta(\{(\xi/t,\zeta/t)\})},
\end{eqnarray*}
which concludes our proof.
\end{proof}

We record a more elegant version of the relation from the above corollary:
\begin{equation}\label{dodici}
\frac{\mu_t(\{\xi\})\nu_t(\{\zeta\})}{\eta_t(\{(\xi,\zeta)\})}=t\frac{\mu(\{\xi/t\})\nu(\{\zeta/t\})}{\eta
(\{(\xi/t,\zeta/t)\})}+1-t.
\end{equation}

\begin{ex}\label{ex43}
We compute a simple example: let $\eta=\frac34\delta_{(1,1)}
+\frac18\delta_{(0,0)}+\frac18\delta_{(1,0)}$. Then
$\mu=\frac18\delta_{0}+\frac78\delta_1$, $\nu=
\frac14\delta_0+\frac34\delta_1$. 
The longest an atom can hope to survive is for as long as $t<4$. Indeed,
$\mu_t(\{0\})=\max\{0,t\mu(\{0\})+1-t\}=\max\{0,1-\frac78t\}$, 
$\mu_t(\{t\})=\max\{0,1-\frac18t\}$, $\nu_t(\{0\})=\max\{0,1-\frac34t\}$,
$\nu_t(\{t\})=\max\{0,1-\frac14t\}$. So if $t<8/7$, then 
$$
\eta_t(\{(0,0)\})=\frac{\frac18\left(1-\frac78t\right)\left(1-\frac34t\right)}{\frac{t}{32}+\frac{1-t}{8}}
=\left(1-\frac78t\right),
$$
if $t<4$, then
$$
\eta_t(\{(t,t)\})=\frac{\frac34\left(1-\frac18t\right)\left(1-\frac14t\right)}{\frac{21}{24}t+\frac34(1-t)}
=\left(1-\frac{t}{4}\right),
$$
and if $t<4/3$, then
$$
\eta_t(\{(t,0)\})=\frac{\frac18\left(1-\frac18t\right)\left(1-\frac34t\right)}{\frac7{32}t+\frac{1-t}{8}}
=\frac18\frac{(8-t)(4-3t)}{4+3t}.
$$
A direct computation shows that the sum of the mass of the three atoms is strictly less than one for 
any $t>1$, so that a nonatomic part occurs immediately after $t=1$, as in the case of free convolution
of measures on $\mathbb R$.

Unlike for free convolution semigroups, the expression for the non-atomic part
of $\eta_t$ is much more unwieldy. Indeed, while in principle formula \eqref{undici} allows
for a direct computation of $G_{\eta_t}$, the actual computation, even for such a simple
measure as the one from Example \ref{ex43}, becomes uncomfortably long. We provide here just
the necessary ingredients: 
The reciprocals of the Cauchy transforms of the marginals at $t=1$ are $G_\mu(z)^{-1}=
z-\frac{7z}{8z-1}$ and $G_\nu(w)^{-1}=w-\frac{3w}{4w-1}$, and of the reciprocal of the
Cauchy transform of $\eta$ is $G_\eta(z,w)^{-1}=\frac{8zw(z-1)(w-1)}{8zw-2z-w+1}$.
For given $t>1$, the subordination functions associated to the two marginals are
$$
\omega_\mu(t,z)=\frac{8z+8-7t+\sqrt{[8z-7t+6-2\sqrt{7(t-1)}][8z-7t+6+2\sqrt{7(t-1)}]}}{16},
$$
and
$$
\omega_\nu(t,w)=\frac{4w+4-3t+\sqrt{[4w+2-3t-2\sqrt{3(t-1)}][4w+2-3t+2\sqrt{3(t-1)}]}}{8}.
$$
Replacing in \eqref{undici} provides the explicit (algebraic) expression for $G_{\eta_t}$.

The case $t=4$, the time when the last atom disappears, provides
\[\omega_\mu(4, z) = \frac{2z - 5 + \sqrt{4z^2 - 22z + 25}}{4}, \quad \omega_\nu(4, w) = \frac{w - 2 + \sqrt{w^2 - 5w + 4}}{2},\]
so that
$$\frac{1}{G_\mu(\omega_\mu(4, z))} = \frac{z - 5 + \sqrt{4z^2 - 22z + 25}}{3},$$
$$\frac{1}{G_\nu(\omega_\nu(4, w))} = \frac{w - 4 + 2\sqrt{w^2 - 5w + 4}}{3}.$$
Then
\begin{eqnarray*}
\lefteqn{\frac{1 - 4}{G_\mu(\omega_\mu(4, z))G_\nu(\omega_\nu(4, w))} =}\\
& &  \frac{-\left(z - 5 + \sqrt{4z^2 - 22z + 25}\right)\left(w - 4 + 2\sqrt{w^2 - 5w + 4}\right)}{3}
\end{eqnarray*}
and
\begin{eqnarray*}
\lefteqn{\frac{4}{G_\eta(\omega_\mu(4, z), \omega_\nu(4, w))}}\\
& = & \left(2z-5+\sqrt{4z^2 - 22z + 25}\right)\left(2z-9+\sqrt{4z^2 - 22z + 25}\right)\\
& &\mbox{}\times\frac{\left(w-2+\sqrt{w^2-5w+4}\right)\left(w-4+\sqrt{w^2-5w+4}\right)}{\left[2\left(2z-5+\sqrt{4z^2-22z+25}\right)-1\right]\left[w-2+\sqrt{w^2-5w+4}-\frac12\right]+\frac32}.
\end{eqnarray*}
Substituting $\frac{1 - 4}{G_\mu(\omega_\mu(4, z))G_\nu(\omega_\nu(4, w))}$ and 
$\frac{4}{G_\eta(\omega_\mu(4, z), \omega_\nu(4, w))}$ into Equation \eqref{undici} produces $G_{\eta_4}(z,w)$.

\end{ex}


\section{Conditionally bi-free analytic subordination}\label{sec:c-bi-free}

In this section, we discuss how the method of Section \ref{sec:bifreesubord} can be used to study the conditionally bi-free additive convolution
in the scalar-valued setting. Motivated by the universal constructions for conditionally free independence \cite{BLS} and bi-free independence 
\cite{BiFree1}, two of us introduced in \cite{GS1} the notion of conditionally bi-free independence for pairs of algebras in the setting of a 
two-state noncommutative probability space $(\mathcal A, \varphi, \psi)$ such that conditionally bi-freeness reduces to bi-freeness when 
$\varphi = \psi$ and reduces to conditionally freeness when only left or only right algebras are considered. For the theoretical definition in 
terms of actions on a reduced free product space, we refer to \cite[Definition 3.4]{GS1}.

Let $(a_1, b_1)$ and $(a_2, b_2)$ be bi-partite self-adjoint two-faced pairs in a two-state $C^*$-noncommutative
probability space $(\mathcal A,\varphi,\psi)$ such that the joint distributions of $(a_j, b_j)$ with respect to 
$(\varphi,\psi)$ coincide with the moments of a pair $(\theta_j, \eta_j)$ of compactly supported probability 
meausres on $\mathbb{R}^2$ via
\[
\varphi(a_j^mb_j^n)=\int_{\mathbb R^2}t^ms^n\,\mathrm{d}\theta_j(t,s) \text{ and } \psi(a_j^mb_j^n)=\int_{\mathbb R^2}t^ms^n\,\mathrm{d}\eta_j(t,s),\quad j=1,2.
\]
If $(a_1, b_1)$ and $(a_2, b_2)$ are conditionally bi-free with respect to $(\varphi, \psi)$, then the joint distribution of 
$(a_1 + a_2, b_1 + b_2)$ is again a pair $(\theta,\eta)$ of compactly supported probability measures on 
$\mathbb R^2$. The pair $(\theta,\eta)$ depends only on the pairs $(\theta_1, \eta_1)$ and $(\theta_2, \eta_2)$, 
and is called the conditionally bi-free additive convolution of $(\theta_1, \eta_1)$ and $(\theta_2, \eta_2)$, denoted 
$(\theta,\eta)=(\theta_1,\eta_1) \bboxplus_{\mathrm{c}}(\theta_2,\eta_2)$. More precisely, $\eta=\eta_1\bboxplus\eta_2$ 
is the bi-free additive convolution of $\eta_1$ and $\eta_2$, and the moments of $\theta$ can be computed using the moments of 
$\theta_1,\theta_2,\eta_1,$ and $\eta_2$ via the formula provided by \cite[Theorem 4.8]{GS1}. As in Section \ref{back}, 
we denote by $\sigma_j, \tau_j$ the marginals of $\theta_j$ and $\mu_j, \nu_j$ the marginals of $\eta_j$ so that 
$\sigma_j$ and $\mu_j$ are the distributions of $a_j$ with respect to $\varphi$ and $\psi$, respectively, and $\tau_j$ and $\nu_j$ are the 
distributions of $b_j$ with respect to $\varphi$ and $\psi,$ respectively. If $\sigma, \tau$ denote the marginals of $\theta$, 
then $(\sigma, \mu)$ is the conditionally free convolution of $(\sigma_1, \mu_1)$ and $(\sigma_2, \mu_2)$, 
denoted $(\sigma,\mu) =(\sigma_1,\mu_1)\boxplus_{\mathrm{c}}(\sigma_2,\mu_2)$, where $\mu=\mu_1\boxplus\mu_2$, 
and similarly $(\tau,\nu)=(\tau_1,\nu_1)\boxplus_{\mathrm{c}}(\tau_2,\nu_2)$, where $\tau=\tau_1\boxplus \tau_2$.

To linearize the conditionally bi-free additive convolution, the partial conditionally bi-free 
$R$-transform was introduced in \cite[Section 5]{GS1} as the analogue of the partial bi-free $R$-transform, 
which is also a function of two complex variables defined on a neighbourhood of zero in $\mathbb{C}^2$. 
To introduce this function, we shall first review how the conditionally free additive convolution is calculated 
(see \cite{BLS,B08}). Let $G_{\sigma_j}$ and $G_{\mu_j}$ be the (one-dimensional) Cauchy transforms of 
$\sigma_j$ and $\mu_j$, respectively, and let $K_{\mu_j}$ be the inverse of $G_{\mu_j}$ on a neighbourhood 
of infinity as in Section \ref{back}. The \textit{conditionally free $R$-transform} of $(\sigma_j, \mu_j)$ is defined by
\[R_{(\sigma_j, \mu_j)}(z) = K_{\mu_j}(z) - \frac{1}{G_{\sigma_j}(K_{\mu_j}(z))}\]
on a small neighbourhood of zero where $K_{\mu_j}$ is defined, and satisfies the relation
\[R_{(\sigma, \mu)}(z) = R_{(\sigma_1, \mu_1)}(z) + R_{(\sigma_2, \mu_2)}(z)\]
for $z$ in a small enough neighbourhood of zero. (While $K_{\mu_j}$ has a simple pole at zero, 
the function $R_{(\sigma_j, \mu_j)}(z)$ extends holomorphically, not meromorphically, in $0$.)

For the pair $(\theta_j, \eta_j)$ of measures on $\mathbb R^2$, let $G_{\theta_j}$ and $G_{\eta_j}$
be the Cauchy transforms of $\theta_j$ and $\eta_j$, respectively. The \textit{partial conditionally bi-free $R$-transform} of $(\theta_j, \eta_j)$ is defined by
\[R_{(\theta_j, \eta_j)}(z, w) = zR_{(\sigma_j, \mu_j)}(z) + wR_{(\tau_j, \nu_j)}(w) + \widetilde{R}_{(\theta_j, \eta_j)}(z, w),\]
where
\begin{eqnarray*}
\lefteqn{\widetilde{R}_{(\theta_j, \eta_j)}(z, w) =}\\
& & \frac{zwG_{\theta_j}(K_{\mu_j}(z), K_{\nu_j}(w))}{G_{\sigma_j}(K_{\mu_j}(z))G_{\tau_j}(K_{\nu_j}(w))G_{\eta_j}(K_{\mu_j}(z), K_{\nu_j}(w))} - \frac{zw}{G_{\eta_j}(K_{\mu_j}(z), K_{\nu_j}(w))},
\end{eqnarray*}
for $z, w$ in a small enough bi-disk centred at zero (see \cite[Corollary 5.7 and Definition 5.8]{GS1}). 
The crucial property of the partial conditionally bi-free $R$-transform is
\[
R_{(\theta, \eta)}(z, w) = R_{(\theta_1, \eta_1)}(z, w) + R_{(\theta_2, \eta_2)}(z, w), \quad |z|+|w|\text{ sufficiently small},
\]
if $(\theta, \eta) = (\theta_1, \eta_1) \bboxplus_{\mathrm{c}} (\theta_2, \eta_2)$. In terms of random variables, 
the conditionally free and bi-free $R$-transforms are defined by exactly the same formulas as above except the 
notations are slightly different, which we summarize as follows. For a random variable $a$ in a two-state noncommutative 
probability space $(\mathcal A,\varphi,\psi)$, let $G_a^\varphi$ and $G_a^\psi$ be the Cauchy transforms of 
$a$ with respect to $\varphi$ and $\psi$, respectively. Then the conditionally free $R$-transform of $a$ is denoted by 
$R_a^{\mathrm{c}}$. Similarly, for a two-faced pair $(a, b)$ in $(\mathcal A,\varphi,\psi)$, let $G_{(a, b)}^\varphi$ 
and $G_{(a,b)}^\psi$ be the two-variable Cauchy transforms of $(a,b)$ with respect to $\varphi$ and $\psi$, 
respectively. The partial conditionally bi-free $R$-transform of $(a, b)$ is denoted by $R^{\mathrm{c}}_{(a,b)}$.

The notion of conditionally free $R$-transform can be generalized to the operator-valued setting as follows 
(see, e.g., \cite{BPV, PV}). Let $(M,E,F,B,\mathcal{D})$ be a $C^*$-$(B, \mathcal{D})$-noncommutative 
probability space. That is, $B \subseteq M$ and $B \subseteq \mathcal{D}$ are unital inclusions of unital 
$C^*$-algebras, $E\colon M \to B$ is a unit-preserving conditional expectation, and $F\colon M \to\mathcal{D}$ is a unital 
$B$-bimodule map. For $X=X^*\in M$, the $B$-valued and $\mathcal{D}$-valued Cauchy transforms of $X$
with respect to $E$ and $F$ are defined by
$G_X(b) = E\left[(b - X)^{-1}\right]$ and $\mathcal{G}_X(b) = F\left[(b - X)^{-1}\right]$, respectively,
for $b\in B,\Im b>0$ or $\|b^{-1}\|<\|X\|^{-1}$. 
As seen in Section \ref{back} when discussing the $B$-valued free $R$-transform, the function 
$G_X$ has a compositional inverse, denoted $K_X$, defined on an open set of $B$ which contains zero 
in its norm-closure. The $\mathcal{D}$-valued conditionally free $R$-transform of $X$ is defined by
\[
R_X^{\mathrm{c}}(b) = K_X(b) - \mathcal{G}_X(K_X(b))^{-1}, \quad \|b\| \text{ small},
\]
where the exponent $-1$ denotes the inverse of $ \mathcal{G}_X(K_X(b))$ in the algebra $\mathcal D$.
The $\mathcal{D}$-valued conditionally free $R$-transform plays the same role as its scalar-valued counterpart: if
$X_1 = X_1^*, X_2 = X_2^* \in M$ are conditionally free over $(B,\mathcal{D})$ with respect to $(E,F)$, then
$R_{X_1 + X_2}^{\mathrm{c}}(b) = R_{X_1}^{\mathrm{c}}(b) + R_{X_2}^{\mathrm{c}}(b)$ for all $b\in B$ of sufficiently small norm. 
We consider now the conditionally free analogue of Lemma \ref{lem1}.

Let $(a, b)$ be a self-adjoint two-faced pair in a two-state $C^*$-noncommutative probability space $(\mathcal A,\varphi,\psi)$. 
Define $X\in M_2(\mathcal A)$ by $X = \begin{bmatrix}
a & 0\\
0 & b
\end{bmatrix}$, let $B=\mathcal{D}=M_2(\mathbb{C})$, and define $M_2(\varphi), M_2(\psi)\colon M_2(\mathcal A)\to M_2(\mathbb{C})$ 
by $M_2(\varphi)=\varphi\otimes\mathrm{Id}_{M_2(\mathbb{C})}$ and $M_2(\psi)=\psi\otimes\mathrm{Id}_{M_2(\mathbb{C})}$, respectively. 
Furthermore, let $G_X$ and $\mathcal{G}_X$ be the Cauchy transforms of $X$ with respect to $M_2(\psi)$ and $M_2(\varphi)$, 
respectively, and let $K_X$ be the compositional inverse of $G_X$. Using the results from Section \ref{sec:bifreesubord} above, we have
\begin{eqnarray*}
\lefteqn{R_X^\mathrm{c}\left(\begin{bmatrix}
z & \zeta\\
0 & w
\end{bmatrix}\right)}\\ 
&= & K_X\left(\begin{bmatrix}
z & \zeta\\
0 & w
\end{bmatrix}\right) - \mathcal{G}_X\left(K_X\left(\begin{bmatrix}
z & \zeta\\
0 & w
\end{bmatrix}\right)\right)^{-1}\\
&= &\begin{bmatrix}
K_a(z) & \frac{-\zeta}{G_{(a, b)}^\psi(K_a(z), K_b(w))}\\
0 & K_b(w)
\end{bmatrix} - \mathcal{G}_X\left(\begin{bmatrix}
K_a(z) & \frac{-\zeta}{G_{(a, b)}^\psi(K_a(z), K_b(w))}\\
0 & K_b(w)
\end{bmatrix}\right)^{-1}\\
&= & \begin{bmatrix}
K_a(z) & \frac{-\zeta}{G_{(a, b)}^\psi(K_a(z), K_b(w))}\\
0 & K_b(w)
\end{bmatrix} - \begin{bmatrix}
G_a^\varphi(K_a(z)) & \frac{\zeta G_{(a, b)}^\varphi(K_a(z), K_b(w))}{G_{(a, b)}^\psi(K_a(z), K_b(w))}\\
0 & G_b^\varphi(K_b(w))
\end{bmatrix}^{-1}\\
& = & \begin{bmatrix}
K_a(z)-\frac{1}{G_a^\varphi(K_a(z))} & \zeta\frac{G_{(a, b)}^\varphi(K_a(z), K_b(w))-G_a^\varphi(K_a(z))G_b^\varphi(K_b(w))}{G_a^\varphi(K_a(z))G_b^\varphi(K_b(w))G_{(a, b)}^\psi(K_a(z), K_b(w))}\\
0 & K_b(w)-\frac{1}{G_b^\varphi(K_b(w))}
\end{bmatrix}\\
&= &\begin{bmatrix}
R_a^\mathrm{c}(z) & \frac{\zeta}{zw}\widetilde{R}_{(a, b)}^{\mathrm{c}}(z, w)\\
0 & R_b^{\mathrm{c}}(w)
\end{bmatrix},
\end{eqnarray*}
a perfect analogue of Lemma \ref{lem1}. We record our conclusion in the following:

\begin{lemma}\label{lem:6.1}
Let $(a,b)$ be a self-adjoint two-faced pair in a two-state $C^*$-noncommutative probability space $(\mathcal A,\varphi,\psi)$. 
Define the $M_2(\mathbb C)$-valued random variable $X=\begin{bmatrix}
a & 0 \\
0 & b
\end{bmatrix}\in M_2(\mathcal A)$ and let $M_2(\varphi), M_2(\psi)$ be as above. Then 
$$
R^{\mathrm{c}}_{X}\left(\begin{bmatrix}
z & \zeta\\
0 & w
\end{bmatrix}\right)=\begin{bmatrix}
R_a^{\mathrm{c}}(z) & \frac{\zeta}{zw}(R^{\mathrm{c}}_{(a,b)}(z,w)-zR^{\mathrm{c}}_a(z)-wR^{\mathrm{c}}_b(w))\\
0 & R^{\mathrm{c}}_b(w)
\end{bmatrix}.
$$
In particular, if $(a_1,b_1)$ and $(a_2,b_2)$ in $(\mathcal A, \varphi, \psi)$ are conditionally bi-free
with respect to $(\varphi, \psi)$, then, under the above notation,
\begin{equation}\label{upperc}
R^{\mathrm{c}}_{X_1+X_2}\left(\begin{bmatrix}
z & \zeta\\
0 & w
\end{bmatrix}\right)=
R^{\mathrm{c}}_{X_1}\left(\begin{bmatrix}
z & \zeta\\
0 & w
\end{bmatrix}\right)+R^{\mathrm{c}}_{X_2}\left(\begin{bmatrix}
z & \zeta\\
0 & w
\end{bmatrix}\right)
\end{equation}
for all $z,w\in\mathbb C$ of sufficiently small absolute value and
all $\zeta\in\mathbb C$.
\end{lemma}

\begin{remark}
As with the bi-free case, $(a_1, b_1)$ and $(a_2, b_2)$ being conditionally bi-free with respect to $(\varphi, \psi)$ does not necessarily imply $X_1$ and $X_2$ are conditionally free with amalgamation over $(M_2(\mathbb{C}), M_2(\mathbb{C}))$ with respect to $(M_2(\varphi), M_2(\psi))$, so that Equation \eqref{upperc} does not extend to arbitrary $2 \times 2$
matrices $\begin{bmatrix}
z & \zeta\\
\zeta' & w
\end{bmatrix}$.
\end{remark}

Next, we discuss how the single-variable analytic subordination functions for $\boxplus$ can be used to compute $\bboxplus_{\mathrm{c}}$. 
Let $(\theta_1,\eta_1)$ and $(\theta_2,\eta_2)$ be as above with marginals $\sigma_j,\tau_j$ for $\theta_j$ and $\mu_j,\nu_j$ for $\eta_j$, 
and let $(\theta,\eta) = (\theta_1,\eta_1)\bboxplus_{\mathrm{c}}(\theta_2,\eta_2)$ with marginals $\sigma,\tau$ for $\theta$ and $\mu,\nu$ for $\eta$. 
Given a Borel probability measure $\lambda$ on $\mathbb R$, recall the function  $h_\lambda$ defined in \eqref{h} by 
$h_\lambda(z)=\frac{1}{G_\lambda(z)}-z$. It was shown in 
\cite[Proposition 3]{B08} that
\begin{equation}\label{eqn:h-function}
h_\sigma(z)=h_{\sigma_1}(\omega_{a_1}(z))+h_{\sigma_2}(\omega_{a_2}(z)),\quad z\in\mathbb{C}^+,
\end{equation}
where $(\sigma,\mu) =(\sigma_1,\mu_1)\boxplus_{\mathrm{c}}(\sigma_2,\mu_2)$ and $\omega_{a_1},\omega_{a_2}$ 
are the single-variable subordination functions related to $\mu_1, \mu_2$ (see \eqref{subord}). 
Of course, we also have $h_\tau(w) = h_{\tau_1}(\omega_{b_1}(w)) + h_{\tau_2}(\omega_{b_2}(w))$, 
where $(\tau, \mu) = (\tau_1, \nu_1) \boxplus_{\mathrm{c}} (\tau_2, \nu_2)$ 
with subordination functions $\omega_{b_1}, \omega_{b_2}$. 
In the following, we present a two-dimensional analogue of Equation \eqref{eqn:h-function}.

Since the partial conditionally bi-free $R$-transform linearizes $\bboxplus_{\mathrm{c}}$, we have
\begin{align*}
&\sum_{j = 1}^2\left[\frac{G_{\theta_j}(K_{\mu_j}(z), K_{\nu_j}(w))}{G_{\sigma_j}(K_{\mu_j}(z))G_{\tau_j}(K_{\nu_j}(w))G_{\eta_j}(K_{\mu_j}(z), K_{\nu_j}(w))} - \frac{1}{G_{\eta_j}(K_{\mu_j}(z), K_{\nu_j}(w))}\right]\\
&= \frac{G_{\theta}(K_{\mu}(z), K_{\nu}(w))}{G_{\sigma}(K_{\mu}(z))G_{\tau}(K_{\nu}(w))G_{\eta}(K_{\mu}(z), K_{\nu}(w))} - \frac{1}{G_{\eta}(K_{\mu}(z), K_{\nu}(w))}.
\end{align*}
Replacing in the above $z$ by $G_{\mu}(z)$ and $w$ by $G_{\nu}(w)$ yields
\begin{align}\label{eqn:theta}
\begin{split}
&\sum_{j = 1}^2\left[\frac{G_{\theta_j}(\omega_{a_j}(z), \omega_{b_j}(w))}{G_{\sigma_j}(\omega_{a_j}(z))G_{\tau_j}(\omega_{b_j}(w))G_{\eta_j}(\omega_{a_j}(z), \omega_{b_j}(w))} - \frac{1}{G_{\eta_j}(\omega_{a_j}(z), \omega_{b_j}(w))}\right]\\
&= \frac{G_{\theta}(z, w)}{G_{\sigma}(z)G_{\tau}(w)G_{\eta}(z, w)} - \frac{1}{G_{\eta}(z, w)}.
\end{split}
\end{align}
For a probability measure $\rho$ on $\mathbb R^2$ with marginals $\rho^{(1)}$ and $\rho^{(2)}$, recall the definition
\eqref{TildeE} of the analytic function $\widetilde{E}_\rho(z,w)=\frac{G_\rho(z, w)}{G_{\rho^{(1)}}(z)G_{\rho^{(2)}}(w)} - 1,$ $(z,w)\in(\mathbb{C}\setminus\mathbb R)^2.$
 Using this function, Equation \eqref{eqn:theta} can be written as
\begin{equation}\label{eqn:theta2}
\frac{\widetilde{E}_\theta(z, w)}{G_\eta(z,w)}=\frac{\widetilde{E}_{\theta_1}(\omega_{a_1}(z),\omega_{b_1}(w))}{G_{\eta_1}(\omega_{a_1}(z),\omega_{b_1}(w))}+\frac{\widetilde{E}_{\theta_2}(\omega_{a_2}(z),\omega_{b_2}(w))}{G_{\eta_2}(\omega_{a_2}(z),\omega_{b_2}(w))}
\end{equation}
as the conditionally bi-free analogue of Equation \eqref{eqn:h-function}. As shown in Section \ref{sec:bifreesubord}, the above equation 
can be viewed as an equality of analytic functions on $\mathbb C^+\times\mathbb C^+$ as soon as we multiply both sides by $G_\eta(z,w)$. Otherwise, 
it is an equality of meromorphic functions, as explained in the comments following the proof of Proposition \ref{Bound}.

Let $(M,E,F,B,\mathcal{D})$ be a $C^*$-$(B,\mathcal{D})$-noncommutative probability space. For $Y = Y^* \in M$, define
\begin{equation}\label{eqn:hfunctionOpV}
h_Y(b) = \mathcal{G}_Y(b)^{-1} - b
\end{equation}
for $b\in B$ with $\Im(b)>0$ (recall that $B\subseteq\mathcal D$). The following operator-valued analogue of Equation \eqref{eqn:h-function} was proved in \cite[Lemma 2.14]{BPV}:
$$
h_{Y_1 + Y_2}(b) = h_{Y_1}(\omega_{Y_1}(b)) + h_{Y_2}(\omega_{Y_2}(b)),
$$
whenever $Y_1,Y_2 \in M$ are self-adjoint, free over $(B,\mathcal{D})$, where $\omega_{Y_j}$ is the $B$-valued subordination function satisfying $G_{Y_1+Y_2}=G_{Y_j}\circ\omega_{Y_j}$.

In view of Lemma \ref{lem:6.1} and Equation \eqref{eqn:hfunctionOpV}, we obtain the following analogue of Remark \ref{rem2}.

\begin{remark}
Under the assumptions and notation of Remark \ref{rem2} and the above discussions, if $(a_1, b_1)$ and $(a_2, b_2)$ in $(\mathcal A,\varphi,\psi)$ are conditionally bi-free with respect to 
$(\varphi,\psi)$, and if $Y_1$ and $Y_2$ in $(M_2(\mathcal A), M_2(\varphi),M_2(\psi),M_2(\mathbb{C}),M_2(\mathbb{C}))$ are conditionally free with respect to $(M_2(\varphi),M_2(\psi))$, 
then
\begin{equation}\label{eqn:GEquation}
\mathcal{G}_{X_1 + X_2}\left(\begin{bmatrix}
z & \zeta\\
0 & w
\end{bmatrix}\right)^{-1}\! + \begin{bmatrix}
z & \zeta\\
0 & w
\end{bmatrix}\! =\! (\mathcal{G}_{X_1} \circ \omega_{Y_1})\left(\begin{bmatrix}
z & \zeta\\
0 & w
\end{bmatrix}\right)^{-1}\!+ (\mathcal{G}_{X_2} \circ \omega_{Y_2})\left(\begin{bmatrix}
z & \zeta\\
0 & w
\end{bmatrix}\right)^{-1}
\end{equation}
for all $z, w \in \mathbb{C}^+$, $\zeta \in \mathbb{C}$. Moreover, as seen in the discussions following Remark \ref{rem2}, $\omega_{Y_j}$ is given by
\[\omega_{Y_j}\left(\begin{bmatrix}
z & \zeta\\
0 & w
\end{bmatrix}\right) = \begin{bmatrix}
\omega_{a_j}(z) & \Pi_j(z, \zeta, w)\\
0 & \omega_{b_j}(w)
\end{bmatrix},\]
where $\omega_{a_j}$ and $\omega_{b_j}$ are the single-variable subordination functions with respect to $\psi$ (i.e. 
$G_{a_1 + a_2}^\psi = G^\psi_{a_j} \circ \omega_{a_j}$ and $G_{b_1 + b_2}^\psi = G^\psi_{b_j} \circ \omega_{b_j}$, 
$j = 1, 2$, and $\Pi_j(z, \zeta, w)$ is the function introduced in Theorem {\rm \ref{Main}}). Consequently, we recover 
Equation \eqref{eqn:theta2} from Equation \eqref{eqn:GEquation}.
\end{remark}

We can now state the following analogue of Theorem \ref{Main} in the 
scalar-valued setting, which easily follows from the above considerations.

\begin{theorem}
Let $(\mathcal A,\varphi,\psi)$ be a two-state $C^*$-noncommutative probability space, and let 
$(a_1, b_1)$ and $(a_2, b_2)$ be self-adjoint two-faced pairs in $(\mathcal A,\varphi,\psi)$ which are 
conditionally bi-free with respect to $(\varphi,\psi)$. Denote $X_j = \begin{bmatrix}
a_j & 0\\
0 & b_j
\end{bmatrix}$, $j = 1, 2$. Then
\begin{eqnarray*}
\lefteqn{M_2(\varphi)\left[\left(\begin{bmatrix}
z & \zeta\\
0 & w
\end{bmatrix}-X_1-X_2\right)^{-1}\right]^{-1} + \begin{bmatrix}
z & \zeta\\
0 & w
\end{bmatrix}=}\\
& & M_2(\varphi)\left[\left(\begin{bmatrix}
\omega_{a_1}(z) & \Pi_1(z,\zeta,w)\\
0 & \omega_{b_1}(w)
\end{bmatrix}-X_1\right)^{-1}\right]^{-1}\\
& & \mbox{}+ M_2(\varphi)\left[\left(\begin{bmatrix}
\omega_{a_2}(z) & \Pi_2(z,\zeta,w)\\
0 & \omega_{b_2}(w)
\end{bmatrix}-X_2\right)^{-1}\right]^{-1}
\end{eqnarray*}
for all $z,w\in\mathbb{C}^+$, $\zeta\in\mathbb{C}$, where $\Pi_1(z,\zeta,w)$ and 
$\Pi_2(z,\zeta,w)$ are the functions introduced in Theorem {\rm \ref{Main}}.
\end{theorem}

\begin{remark}
We conclude this section with a remark. In noncommutative probability theory, there 
is another notion of independence, called monotonic independence, introduced by Muraki which, 
together with tensor, free, Boolean, and anti-monotonic independences, form the only five 
notions of natural independence. In the operator-valued setting, if $X_1=X_1^*$ and $X_2=X_2^*$ 
in $(M,E,B)$ are monotonically independent with amalgamation over $B$, then 
$G_{X_1+X_2}(b)=G_{X_1}(G_{X_2}(b)^{-1})$. Now, if $(a_1,b_1)$ and $(a_2, b_2)$ 
are self-adjoint two-faced pairs in $(\mathcal A, \varphi)$ and if we define $X_j = \begin{bmatrix}
a_j & 0\\
0 & b_j
\end{bmatrix}$ for $j = 1, 2$ as before, then by computing
\[G_{X_1 + X_2}\left(\begin{bmatrix}
z & \zeta\\
0 & w
\end{bmatrix}\right) \text{ and } G_{X_1}\left(G_{X_2}\left(\begin{bmatrix}
z & \zeta\\
0 & w
\end{bmatrix}\right)^{-1}\right),\]
and comparing the $(1,2)$ entries, we obtain an expression for $G_{(a_1+a_2, b_1+b_2)}(z,w)$ in 
terms of the Cauchy transforms of the two pairs and the marginals. This has been shown in
\cite{GHS} to lead to one of the two natural notions of bi-monotonic independence, generalizing 
monotonic independence to the two-faced setting. 
\end{remark}


\section{No conditional expectations of the resolvent}\label{neg}

The analytic subordination result of Biane is stronger than 
the result stated in \eqref{subord}: it is shown in \cite{Biane1}
that if $a_1,a_2$ are free self-adjoint random variables in the
{\em tracial} $C^*$-noncommutative probability space 
$(\mathcal A,\varphi)$, then
$$
E_{\mathbb C[a_j]}\left[(z-a_1-a_2)^{-1}\right]=(\omega_{a_j}(z)-a_j
)^{-1},\quad z\in\mathbb C^+,
$$
where $E_{\mathbb C[a_j]}$ denotes the unique trace-preserving 
conditional expectation from the von Neumann algebra 
generated by $a_1$ and $a_2$ onto the von Neumann
algebra generated by $a_j$. Voiculescu generalized this
result to self-adjoint random variables which are free
with amalgamation with respect to a trace-preserving
conditional expectation. However, in order to prove
formula \eqref{subord} alone, both in its scalar- and 
operator-valued version, only analytic function
theory methods are needed, as shown in 
\cite{BBSubord,BMS}. It is remarkable that one can
use this same analytic functions machinery to 
prove Biane's result, at the cost of an amplification
to $3\times 3$ matrices. This method seemed
particularly well suited to prove a similar result for the 
expectation of the product of the resolvents of sums of bi-free, 
bi-partite variables. Unfortunately, that turns out not to be the case. 
In order to explain why Biane's result cannot be fully generalized to 
bi-free variables, we shall give an outline 
of this procedure below.

Recall that if $(\mathcal A,\varphi)$ is a tracial
$W^*$-noncommutative probability space and $B\subseteq
\mathcal A$ is a von Neumann subalgebra, then there
exists a unique trace-preserving conditional expectation
$E\colon\mathcal A\to B$. This expectation is defined
via the following relation: for any $x\in\mathcal A$, 
$E[x]$ is the unique element in $B$ so that $\varphi(x\xi^*)
=\varphi(E[x]\xi^*)$ for all $\xi\in B$. Clearly, given the
hypothesis of weak${}^*$-continuity and faithfulness on 
$\varphi$, it is enough to verify the equality $\varphi(x\xi^*)
=\varphi(E[x]\xi^*)$ for all elements $\xi$ in a subset of $B$
whose linear span is dense in $B$. Thus, in order to 
prove the relation $E_{\mathbb C[a_j]}
\left[(z-a_1-a_2)^{-1}\right]=(\omega_{a_j}(z)-a_j)^{-1},$
it suffices to show that for any $v\in\mathbb C[a_j],v>0$,
we have 
$$
\varphi\left((z-a_1-a_2)^{-1}v\right)=
\varphi\left((\omega_{a_j}(z)-a_j)^{-1}v\right).
$$
In order to do this, we use a linearization trick similar
to the one used in \cite{BMS} (which originates in Anderson's paper \cite{A})
and Lemma \ref{Matr-Bi-Free}, with the right variable 
equal to zero.
Consider $A_1=\begin{bmatrix}
0 & 0 & 0\\
0 & 0 & -v\\
0 & -v & a_1
\end{bmatrix}\in M_3(\mathbb C[a_1]), A_2=\begin{bmatrix}
0 & 0 & 0\\
0 & 0 & 0\\
0 & 0 & a_2
\end{bmatrix}\in M_3(\mathbb C[a_2])$ (recall that $v$ is 
an arbitrary positive element in the von Neumann algebra 
generated by $a_1$). Lemma \ref{Matr-Bi-Free} implies 
that $A_1$ and $A_2$ are free with amalgamation over $M_3(\mathbb C)$
with respect to $M_3(\varphi):=\varphi\otimes{\rm Id}_{M_3(\mathbb C)}$.
According to  \cite[Theorem 2.7]{BMS}, relation \eqref{subord}
holds for the $M_3(\mathbb C)$-valued Cauchy transforms of
$A_1,A_2$ and $A_1+A_2$. We have
\begin{eqnarray}
\lefteqn{G_{A_1+A_2}\left(\begin{bmatrix}
0 &1 & 0\\
1 & 0 & 0\\
0 & 0 & z
\end{bmatrix}\right)=
M_3(\varphi)\left(\left(
\begin{bmatrix}
0 &1 & 0\\
1 & 0 & 0\\
0 & 0 & z
\end{bmatrix}-A_1-A_2\right)^{-1}\right)}\nonumber\quad\quad\quad\quad\quad\quad\quad\\
& =  &
\begin{bmatrix}
\varphi(v(z-a_1-a_2)^{-1}v) & 1 & -\varphi(v(z-a_1-a_2)^{-1})\\
1 & 0 & 0 \\
-\varphi((z-a_1-a_2)^{-1}v) & 0 & \varphi((z-a_1-a_2)^{-1})
\end{bmatrix}.\label{inverse}
\end{eqnarray}
For simplicity, we denote $R=(z-a_1-a_2)^{-1}$, so that $G_{a_1+a_2}(z)=\varphi(R)$.
Then
$$
G_{A_1+A_2}\left(\begin{bmatrix}
0 &1 & 0\\
1 & 0 & 0\\
0 & 0 & z
\end{bmatrix}\right)^{-1}=\begin{bmatrix}
0 & 1 & 0\\
1 & \frac{\varphi(vR)\varphi(Rv)}{\varphi(R)}-\varphi(vRv) & \frac{\varphi(vR)}{\varphi(R)} \\
0 & \frac{\varphi(Rv)}{\varphi(R)} & \frac{1}{\varphi(R)}
\end{bmatrix}.
$$
Theorem 2.7 of \cite{BMS} guarantees (through purely 
function-theoretic arguments) the existence of subordination functions 
$\omega_{A_1}$ and $\omega_{A_2}$ satisfying \eqref{subord}.
This relation implies via a few arithmetic manipulations and a few applications 
of the identity principle for analytic functions, that there are functions
$\theta_2,\theta_3,\tau_2,\tau_3$ depending analytically on $z$ such that
$$
\omega_{A_1}\left(\begin{bmatrix}
0 &1 & 0\\
1 & 0 & 0\\
0 & 0 & z
\end{bmatrix}\right)=\begin{bmatrix}
0 & 1 & 0\\
1 & \theta_2(z) & \theta_3(z) \\
0 & \theta_3(z) & \omega_{a_1}(z)
\end{bmatrix},
$$
$$\omega_{A_2}\left(\begin{bmatrix}
0 &1 & 0\\
1 & 0 & 0\\
0 & 0 & z
\end{bmatrix}\right)=\begin{bmatrix}
0 & 1 & 0\\
1 & \tau_2(z) & \tau_3(z) \\
0 & \tau_3(z) & \omega_{a_2}(z)
\end{bmatrix},
$$
(the functions $\omega_{a_j}$ are the subordination functions 
from formula \eqref{subord} associated to $a_j,j=1,2$).
The subordination relations $G_{A_1+A_2}=G_{A_2}\circ\omega_{A_2}$ and $G_{a_1+a_2}=G_{a_2}\circ\omega_{a_2}$
translate into
$$
\begin{bmatrix}
\varphi(vRv) & 1 & -\varphi(vR)\\
1 & 0 & 0 \\
-\varphi(Rv) & 0 & {\varphi(R)}
\end{bmatrix}=\begin{bmatrix}
\tau_3(z)^2\varphi(R)-\tau_2(z) & 1 & -\tau_3(z)\varphi(R)\\
1 & 0 & 0 \\
-\tau_3(z)\varphi(R) & 0 & \varphi(R)
\end{bmatrix}
$$
Thus, $\tau_3(z)=\frac{\varphi(vR)}{\varphi(R)}$. Together with relation \eqref{subord} applied to the matrix-valued functions, this provides us with the equality $\theta_3(z)=0$
(and, as an added bonus, $\tau_2(z)=\frac{\varphi(vR)^2}{\varphi(R)}-\varphi(vRv)$).
The subordination relations corresponding to $A_1$ and $a_1$ yield
$$
\begin{bmatrix}
\varphi(vRv) & 1 & -\varphi(vR)\\
1 & 0 & 0 \\
-\varphi(Rv) & 0 & {\varphi(R)}
\end{bmatrix}=\begin{bmatrix}
\varphi\left(\frac{v^2}{\omega_{a_1}(z)-a_1}\right)-\theta_2(z) & 1 & -\varphi\left(\frac{v}{\omega_{a_1}(z)-a_1}\right)\\
1 & 0 & 0 \\
-\varphi\left(\frac{v}{\omega_{a_1}(z)-a_1}\right) & 0 & \varphi\left(\frac{1}{\omega_{a_1}(z)-a_1}\right)
\end{bmatrix},
$$
The equality of $(1,3)$ entries completes the proof of Biane's result. As an added bonus, traciality of $\varphi$ allows us to conclude also that $\theta_2(z)=0$, which
determines $\omega_{A_1},\omega_{A_2}$ on our variable.

Based on Lemma \ref{lem1} and Remark \ref{rem2}, it is 
tempting to use the same trick in order to find 
$E_{\mathbb C[a_j,b_j]}\left[(z-a_1-a_2)^{-1}(w-b_1-b_2)^{-1}\right]$
for $(a_1,b_1)$ and $(a_2,b_2)$ in $\mathcal A^2$ bi-free with
respect to $\varphi$ and bi-partite (that is, $a_jb_j=b_ja_j$, $j=1,2$).
Consider the $6\times6$ matrix
$$
V=\begin{bmatrix}
0 & 1 & 0 & 0 & 0 & 0\\
1 & 0 & v & 0 & 0 & 0\\
0 & v & z-a_1-a_2& 0 & 0 & 1\\
0 & 0 & 0 & 0 & 1 & 0\\
0 & 0 & 0 & 1 & 0 & u\\
0 & 0 & 0 & 0 & u & w-b_1-b_2
\end{bmatrix},
$$
where $v\in\mathbb C[a_1],u\in\mathbb C[b_1]$ are both strictly 
positive. Lemma \ref{Matr-Bi-Free} guarantees that 
$$
(A_1,B_1)=\left(\begin{bmatrix}
0 & 0 & 0 \\
0 & 0 & -v \\
0 & -v & a_1
\end{bmatrix},\begin{bmatrix}
0 & 0 & 0\\
0 & 0 & -u\\
0 & -u & b_1
\end{bmatrix}\right)$$ and 
$$(A_2,B_2)=\left(\begin{bmatrix}
0 & 0 & 0 \\
0 & 0 & 0\\
0 & 0 & a_2
\end{bmatrix},\begin{bmatrix}
0 & 0 & 0\\
0 & 0 & 0\\
0 & 0 & b_2
\end{bmatrix}\right)
$$ are bi-free with amalgamation over $M_3(\mathbb C)$. For simplicity, 
let $Z=\begin{bmatrix}
0 &1 & 0\\
1 & 0 & 0\\
0 & 0 & z
\end{bmatrix}$, $W=\begin{bmatrix}
0 &1 & 0\\
1 & 0 & 0\\
0 & 0 & w
\end{bmatrix},$ and $e_{3,3}=\begin{bmatrix}
0 &0 & 0\\
0 & 0 & 0\\
0 & 0 & 1
\end{bmatrix}$
Proposition \ref{Main} and Remark \ref{rem2} apply to $X_j=\begin{bmatrix}
A_j & 0\\
0 & B_j
\end{bmatrix}$, $j=1,2$, and the scalar matrix $\begin{bmatrix}
Z & e_{3,3}\\
0 & W
\end{bmatrix}$.
On the other hand, 
inverting the matrix $V$, we obtain on the two $3\times 3$ diagonal
blocks precisely the formula from \eqref{inverse} and its analogue
for $b_1,b_2,w$. In the upper right $3\times3$ corner, we obtain the
matrix 
$$
\begin{bmatrix}
v(z-a_1-a_2)^{-1}(w-b_1-b_2)^{-1}u & 0 & -v(z-a_1-a_2)^{-1}(w-b_1-b_2)^{-1}\\
0 & 0 & 0\\
(z-a_1-a_2)^{-1}(w-b_1-b_2)^{-1}u & 0 & (z-a_1-a_2)^{-1}(w-b_1-b_2)^{-1}
\end{bmatrix}
.$$
If $\varphi$ were tracial, applying $\varphi$ on the above and comparing with
the corresponding matrix entry from $G_{X_1}\left(\omega_{X_1}\left(\begin{bmatrix}
Z & e_{3,3}\\
0 & W
\end{bmatrix}\right)\right)$ would provide the bi-free analogue of Biane's 
result. However, it turns out that $\varphi$ is tracial only in the relatively trivial
case in which the two faces are independent. We emphasize that a formula 
giving $E_{\mathbb C[a_j,b_j]}\left[(z-a_1-a_2)^{-1}(w-b_1-b_2)^{-1}\right]$ as a product
of resolvents of $a_j$ and $b_j$ would imply that
$\varphi(v(z-a_1-a_2)^{-1}(w-b_1-b_2)^{-1}u)=\varphi(uv(z-a_1-a_2)^{-1}(w-b_1-b_2)^{-1})$
for $v\in\mathbb C[a_j],u\in\mathbb C[b_j]$.

\begin{theorem}\label{trace}
Let $(a_1, b_1)$ and $(a_2, b_2)$ be pairs of self-adjoint operators that are bi-free in a $^*$-noncommutative probability space $(\mathcal A, \varphi)$.  Suppose that $\tau := \varphi|_{\alg(a_1, a_2, b_1, b_2)}$ is tracial and that for each $k \in \{1,2\}$ there does not exists $\alpha_k, \beta_k \in \mathbb R$ such that $(\varphi(a_k^n), \varphi(b_k^n)) = (\alpha_k^n, \beta_k^n)$ for all $n \in \mathbb N$ (i.e. neither pair is scalars in distribution).  Then $\alg(a_1, a_2)$ and $\alg(b_1, b_2)$ are independent.  In particular, $\varphi$ decomposes as the tensor product of tracial states on $\alg(a_1, a_2)$ and $\alg(b_1, b_2)$.
\end{theorem}
\begin{proof}
By the combinatorial theory of bi-free independence (see \cite{CNS}) it suffices to prove the following:  for all $n \in \mathbb N$, for all non-constant $\chi : \{1,\ldots, n\} \to \{\ell, r\}$, and for all $k\in \{1, 2\}$ we have
\[
\kappa_\chi(c_1, \ldots, c_n) = 0
\]
where $c_m = a_k$ if $\chi(m) = \ell$ and $c_m = b_k$ if $\chi(m) = r$.  We will only verify the above when $k = 1$ as the case $k = 2$ follows by symmetry.  We proceed by induction on $n$.  

As there does not exists $\alpha, \beta \in \mathbb R$ such that $(\varphi(a_2^n), \varphi(b_2^n)) = (\alpha^n, \beta^n)$ for all $n \in \mathbb N$, there exists $n_1, n_2 \in \mathbb N$ such that $\varphi(a_2^{n_1+n_2}) \neq \varphi(a_2^{n_1}) \varphi(a_2^{n_2})$ or $\varphi(b_2^{n_1+n_2}) \neq \varphi(b_2^{n_1}) \varphi(b_2^{n_2})$.  We will assume that $\varphi(a_2^{n_1+n_2}) \neq \varphi(a_2^{n_1}) \varphi(a_2^{n_2})$ as the other case will follow by similar arguments.

The case $n = 1$ is trivial so we begin with the case $n = 2$.  Here $(\chi(1), \chi(2)) \in \{(\ell, r), (r, \ell)\}$.   By bi-freeness and traciality, we know that
\begin{align*}
\varphi(a_2^{n_1+n_2}) \varphi(a_1 b_1) &= \varphi(a_2^{n_1+n_2} a_1 b_1) \\
&= \varphi(a_2^{n_2} a_1 b_1a_2^{n_1}) \\
&= \varphi(a_2^{n_1+n_2}) \varphi(a_1) \varphi(b_1) + \varphi(a_2^{n_1}) \varphi(a_2^{n_2}) \kappa_{(\ell, r)}(a_1, b_1).
\end{align*}
Thus, as
\[
\kappa_{(\ell, r)}(a_1, b_1) = \varphi(a_1 b_1)  - \varphi(a_1) \varphi(b_1) 
\]
we obtain that
\[
\varphi(a_2^{n_1+n_2})\kappa_{(\ell, r)}(a_1, b_1) = \varphi(a_2^{n_1}) \varphi(a_2^{n_2}) \kappa_{(\ell, r)}(a_1, b_1).
\]
As $\varphi(a_2^{n_1+n_2}) \neq \varphi(a_2^{n_1}) \varphi(a_2^{n_2})$, this implies $\kappa_{(\ell, r)}(a_1, b_1) = 0$.  Similarly, 
\begin{align*}
\varphi(a_2^{n_1+n_2}) \varphi(b_1 a_1) &= \varphi(a_2^{n_1+n_2} b_1 a_1) \\
&= \varphi(a_2^{n_2} b_1 a_1 a_2^{n_1}) \\
&= \varphi(a_2^{n_1+n_2}) \varphi(a_1) \varphi(b_1) + \varphi(a_2^{n_1}) \varphi(a_2^{n_2}) \kappa_{(r, \ell)}(b_1, a_1).
\end{align*}
Thus the same argument implies $\kappa_{(r, \ell)}(b_1, a_1) = 0$.  

For the inductive step, suppose we have verified the result for $n-1$ for some $n\in\mathbb N$.  Let  $\chi : \{1,\ldots, n\} \to \{\ell, r\}$ be non-constant  and let where $c_m = a_1$ if $\chi(m) = \ell$ and $c_m = b_1$ if $\chi(m) = r$.  Let $m_1 = |\chi^{-1}(\{\ell\})|$ and let $m_2 = |\chi^{-1}(\{r\})|$.   By bi-freeness and traciality, we know that
\begin{align*}
\varphi(a_2^{n_1+n_2}) \varphi(c_1 \cdots c_n) &= \varphi(a_2^{n_1+n_2} c_1 \cdots c_n) \\
&= \varphi(a_2^{n_2} c_1 \cdots c_n a_2^{n_1}) \\
&= \varphi(a_2^{n_1+n_2}) \varphi(a_1^{m_1}) \varphi(b_1^{m_2}) + \varphi(a_2^{n_1}) \varphi(a_2^{n_2})\kappa_\chi(c_1, \ldots, c_n)
\end{align*}
(where we have used the induction hypothesis to deduce any cumulant involving $a_1$ and $b_1$ of length at most $n-1$ is zero).  As 
\[
\varphi(c_1 \cdots c_n) = \varphi(a_1^{m_1}) \varphi(b_1^{m_2}) + \kappa_\chi(c_1, \ldots, c_n)
\]
(where we have used the induction hypothesis to deduce any cumulant involving $a_1$ and $b_1$ of length at most $n-1$ is zero), we obtain that
\[
\varphi(a_2^{n_1+n_2})\kappa_\chi(c_1, \ldots, c_n) =  \varphi(a_2^{n_1}) \varphi(a_2^{n_2})\kappa_\chi(c_1, \ldots, c_n).
\]
As $\varphi(a_2^{n_1+n_2}) \neq \varphi(a_2^{n_1}) \varphi(a_2^{n_2})$, this implies $\kappa_\chi(c_1, \ldots, c_n) = 0$.  Hence the result follows.
\end{proof}

\end{document}